\documentclass{amsart}

\usepackage[pagebackref=false,colorlinks=true, pdfstartview=FitV, linkcolor=blue,citecolor=red, urlcolor=blue]{hyperref}

\usepackage{amsmath,amsthm,amssymb,verbatim,mathrsfs}%
\usepackage{color}
\usepackage{hyperref}
\usepackage{url}

\newtheorem{thm}{Theorem}[section]
\newtheorem{cor}[thm]{Corollary}
\newtheorem{lem}[thm]{Lemma}
\newtheorem{prop}[thm]{Proposition}

\newtheorem{theorema}{Theorem}

\theoremstyle{definition}

\theoremstyle{definition}

\newtheorem{rem}[thm]{Remark}

\newcommand\be{\begin{equation}}
\newcommand\ee{\end{equation}}
\newcommand\bee{\begin{equation*}}
\newcommand\eee{\end{equation*}}
\newcommand\ben{\begin{enumerate}}
\newcommand\een{\end{enumerate}}

\def\CC{\ensuremath {{\mathbb C}}}
\def\H{\ensuremath {{\mathscr H}}}
\def\I{\ensuremath {{\mathscr I}}}

\def\R{\ensuremath {{\mathbb R}}}
\def\F{\ensuremath {{\mathscr F}}}

\def\A{\ensuremath {{\mathbb A}}}

\def\P{\ensuremath {{\mathscr P}}}

\def\L{\ensuremath {\mathscr L}}
\def\Z{\ensuremath {{\mathbb Z}}}
\def\D{\ensuremath {\mathscr D}}
\def\C{\ensuremath {\mathscr C}}
\def\U{\ensuremath {\mathscr U}}

\def\K{\ensuremath {\mathcal K}}

\def\M{\ensuremath {\mathscr M}}
\def\J{\ensuremath {\mathscr J}}

\def\o{\ensuremath {\mathfrak o}}
\def\O{\ensuremath {\mathscr O}}

\def\a{\ensuremath {\mathfrak a}}
\def\tr{\ensuremath {\textnormal{tr}}}
\def\vol{\ensuremath {\textnormal{vol}}}
\def\GL{\ensuremath {\mathrm{GL}}}

\def\bs{\ensuremath {\backslash}}
\def\ac{\ensuremath {\mathrm{ac}}}

\def\disc{\ensuremath {\mathrm{disc}}}

\numberwithin{equation}{section}

\title[A weighted invariant trace formula]{A weighted invariant trace formula}
\author{Tian An Wong}
\email{tiananw@umich.edu}
\address{University of Michigan, 4901 Evergreen Road, Dearborn, MI 48128}

\begin{document}
\begin{abstract}
This paper begins a new approach to the $r$-trace formula, without removing the nontempered contribution to the spectral side. We first establish an invariant trace formula whose discrete spectral terms are weighted by automorphic $L$-functions. This involves extending the results of Finis, Lapid, and M\"uller on the continuity of the coarse expansion of Arthur's noninvariant trace formula to the refined expansion, and then to the invariant trace formula, while incorporating the use of basic functions at unramified places. 
\end{abstract}

\maketitle

\tableofcontents \addtocontents{toc}{\protect\setcounter{tocdepth}{1}}

\section{Introduction}

\subsection{Motivation}

The Arthur-Selberg trace formula is one of the major tools in the theory of automorphic forms and harmonic analysis. Given a reductive group $G$ over a number field $F$, there is a linear form $J$ on $C_c^\infty(G(\A_F)^1)$ with parallel spectral and geometric expansions which constitute the trace formula. The monumental work of Arthur established the stabilization of the trace formula,  which in turn has led to the endoscopic classification of automorphic representations of various classical groups. In order to do so, one first makes the trace formula invariant, expressing $J$ as a linear combination of invariant distributions on $G$; its stabilization in turn depends crucially upon the Fundamental Lemma. To gain deeper knowledge of Langlands' principle of functoriality, and in particular beyond the endoscopic cases, it is important to establish further refinements of the trace formula.  Langlands' original investigation \cite{BE} analyzed a limiting form of the trace formula, whereby the spectral terms would be weighted by the order of the poles of the relevant $L$-functions (or their residues) at $s=1$. In order for the limit to exist, it was understood that the contribution of nontempered automorphic representations to the trace formula would first have to be removed by some means. In \cite{FLN} it was suggested that a Poisson summation formula be applied to the elliptic terms on the geometric side of the trace formula, where the dual sum would be used to cancel the nontempered contribution. 

To date, this method has only been made to work in the setting of GL$_2(\mathbb Q)$ and the standard representation \cite{Alt1}, and even there requiring additional methods from analytic number theory that appear to be prohibitive in higher rank \cite{GKM}. Moreover, as Arthur has shown in \cite{artstrat}, even if we take for granted the generalization of \cite{Alt1} to general $G$, it remains unclear how one should match the terms in the dual sum to the contribution of nontempered terms, let alone how the limiting form of the trace formula might be studied. In hindsight, in attempting to lay out a path beyond endoscopy, Langlands had uncovered numerous problems that have to be solved, each one difficult in its own right \cite{BE}. It is perhaps because of these difficulties that other new directions have unfolded in the last decade under the banner of `Beyond Endoscopy,' divergent from Langlands' vision of using the Arthur-Selberg trace formula; these exciting new directions make points of contact with other important problems in the theory of automorphic forms and raise interesting questions in arithmetic geometry. That being said, in this paper we shall stay within the ambit of the aforementioned works, introducing a new approach while still falling in line with Arthur's interpretation of Langlands' proposal \cite{problems}. 

In Arthur's formulation, two further refinements should follow the stabilization of the trace formula. The first refinement, now called the $r$-trace formula, comes in the form of a trace formula whose spectral terms are weighted by factors related to the poles of automorphic $L$-functions $L(s,\pi,r)$. Implicit in this is an extension of the trace formula to a class of noncompactly supported test functions that we denote by $\C^\circ(G(\A)^1)$, where $\A=\A_F$, defined in \eqref{Ccirc} and preceding. This latter extension was established by Finis, Lapid, and M\"uller \cite{FLM,FLss,FL} for the coarse expansion of the noninvariant form $J$, as a distribution on the group $G(\A)^1$, 
\[
J(\dot f) = \sum_{\chi\in\mathfrak X}J_\chi(\dot f) = \sum_{\o\in \O}J_\o(\dot f),\qquad \dot f\in\C^\circ(G(\A)^1).
\]
Among such test functions is a distinguished family of functions, now referred to as basic functions, which can be used to weight the cuspidal spectral terms with the associated automorphic $L$-functions. As the $r$-trace formula should be weighted by coefficients that are nonzero if and only if $L(s,\pi,r)$ has a pole at $s=1$ (such as the residue), we can view the incorporation of the basic function as a step towards the $r$-trace formula. We note that up until now, basic functions have been little studied in the context of the Arthur-Selberg trace formula. While the geometry of basic functions have been the subject of intense study, the behaviour of their orbital integrals is less understood. Of course, the coarse expansion of $J(\dot f)$ is only the beginning of the story; the objective of this paper is to arrive at an invariant trace formula that incorporates the use of the basic function, and also to lay the groundwork for its stabilization in \cite{wstf1}. 

Crucially, our approach does not presume the need to remove the nontempered contribution to the trace formula, which was the core difficulty encountered in previous attempts to carry out Langlands' proposal with the Arthur-Selberg trace formula. With this modest but significant change in strategy, rather than having to remove the nontempered contribution and then take a limiting form of the trace formula, we are able to take an unconditional step forward by obtaining expansions for the trace formula involving automorphic $L$-functions. We give this intermediate distribution the uninspired name of a {\em weighted trace formula}. This will lead us to the problem of obtaining the meromorphic continuation of the resulting distribution to the point $s=1$. Note that even $L$-functions of nontempered automorphic representations are expected to have meromorphic continuation, so a priori this is not an obstacle as it was in the limit of trace formulas. The latter method required that all the relevant $L$-functions have analytic continuation to the right-half plane Re$(s)>1$, so the nontempered contribution had to be removed first. Incidentally, this was the original motivation for considering relative trace formulas instead, but that perspective leads us down a significantly different path. 

Recall that our goal is to establish an $r$-trace formula. The results of this paper together with the sequel \cite{wstf1} represent progress in that direction. Taken together, we shall have in our hands a stable distribution that incorporates the data of the basic function, and therefore automorphic $L$-functions, from which point the task is to prove meromorphic continuation. With that, the residual distribution at $s=1$ will be an $r$-stable trace formula. One benefit of these unconditional steps forward is that they place us on solid ground, and gradually yield insight into how this continuation might be obtained. Indeed, it seems at the moment that we shall find certain overlap with recent work of Hoffmann on prehomogenous vector spaces and Ng\^{o} on Hankel transforms, both in the context of the geometric side of the trace formula. The situation there might not yet be clear, but more and new tools appear to be available there, which were not  in earlier works. As Arthur's pioneering work on the trace formula has taught us, we can expect the path towards the $r$-trace formula and its later refinement, the primitive trace formula, to be a long and arduous one. The attendant problems will not be solved in a single work, but piece by piece. It is with this long view of the old road to functoriality that we undertake this task. 

\subsection{Main results}

In this paper, we make use of basic functions to establish a weighted invariant trace formula, whose cuspidal spectral terms are weighted by automorphic $L$-functions. Fixing a central induced torus $Z$ of $G$ with an automorphic character $\zeta$ of $Z(F)\bs Z(\A)$, we let $V$ be a large finite set of valuations of $F$ outside of which $G$ and $\zeta$ are unramified. Let $G_V = G(F_V) = \prod_{v\in V} G(F_v)$ and $G^V =G(\A^V) =  \prod_{v\not\in V}' G(F_v)$, where the latter is the usual restricted direct product. The main technical difficulty that we encounter is that Arthur's stabilization of the trace formula is valid only for test functions of the form 
\[
\dot f = f \times u^V
\]
as in \eqref{frs}, where $f$ is a $\zeta^{-1}$-equivariant function {in $C_c^\infty(G(F_V))$} and $u^V$ is the unit element of the $\zeta^{-1}$-equivariant Hecke algebra $\H(G,V,\zeta)$ of $G(\A^V)$ as in \eqref{heck}. In order to properly weight the trace formula, we require instead test functions of the form
\[
f^r_s = f\times b^V
\]
where $f$ is a $\zeta^{-1}$-equivariant function {in $\C^\circ(G(F_V))$} as defined in Section \ref{defs}, and  $b^V$ is the basic function which, as we recall in Section \ref{basic}, does not have compact support. It depends on an irreducible complex finite-dimensional representation $r$ of the $L$-group $^LG$ of $G$, and a complex number $s$ with Re$(s)$ large enough, which we shall assume to be fixed throughout this paper. Note that for our applications, $f$ can in fact taken to be compactly-supported. 

In any case, we have to take the coarse expansion of the noninvariant linear form $J$ as our starting point, and begin the process of refinement there. More precisely, we shall first establish a refined expansion for the noninvariant linear form $J$ applied to $f^r_s$, namely, {for Re$(s)$ large enough, we have}
\begin{align*}
J(f^r_s) &= \sum_{M\in\L}|W^M_0||W^G_0|^{-1} \int_{\Pi(M,V,\zeta)}a^M_{r,s}(\pi)J_M(\pi,f) d\pi\\
&=\sum_{M\in\L}|W^M_0||W^G_0|^{-1} \sum_{\gamma\in\Gamma(M,V,\zeta)}a^M_{r,s}( \gamma)J_M(\gamma,f)
\end{align*}
which, taking $b^V$ to be fixed, we may view as a linear form on $\C^\circ(G,V,\zeta)$. The local distributions $J_M(\pi,f)$ and $J_M(\gamma,f)$ occurring on either side are the usual weighted characters and weighted orbital integrals appearing in Arthur's trace formula. The global coefficients $a^{M}_{r,s}(\pi)$ and $a_{r,s}^{M}(\gamma)$, on the other hand, are now weighted forms of the coefficients $a^M(\pi)$ and $a^M(\gamma)$ that occur in the usual trace formula, where most importantly the spectral coefficient $a^M_{r,s}(\pi)$ now carries the data of the unramified automorphic $L$-function $L^V(s,\pi,r)$. Part of the refined spectral expansion was already obtained in \cite{FLM}, so the bulk of the work falls on refining the geometric expansion. To do so, we revisit Arthur's original arguments, modifying them where the support of the test function $f^r_s$ is relevant. As in Arthur's work, the core argument relies on a study of the unipotent distributions, where the data of the basic function is abstractly incorporated into the geometric coefficient $a^M_{r,s}(\gamma)$ by an invariance argument. And as with the case of $a^M(\gamma)$, our new coefficients $a^M_{r,s}(\gamma)$ are also only explicit for semisimple $\gamma$. The explicit determination of these coefficients for general $\gamma$, we note, is the goal of a recent program of Hoffmann that relates them to zeta functions of prehomogeneous vector spaces. We shall return to this point in a later paper. 

Having established the refined expansion of $J(f^r_s)$, we then proceed to make this form invariant. Let $\C^\circ(G,V,\zeta)$ be the space of $\zeta^{-1}$-equivariant functions in $\C^\circ(G(F_V))$ {(see Section \ref{defs} for the precise definition).} The main result of this paper is an invariant trace formula that is valid for the test functions $f\times b^V$, which we refer to here as the weighted invariant trace formula.

\begin{theorema}
\label{theorem1}
Let $s\in\CC$ with $\mathrm{Re}(s)$ large enough. The linear form 
\[
I^r_s(f) = I({f}^r_s), \qquad f\in \C^\circ(G,V,\zeta)
\]
has the parallel expansions
\begin{align}
\label{theorema}
&  \sum_{M\in\L}|W^M_0||W^G_0|^{-1}\int_{\Pi(M,V,\zeta)}a^{M}_{r,s}(\pi)I_M(\pi,f)d\pi\\
&= \sum_{M\in\L}|W^M_0||W^G_0|^{-1}\sum_{\gamma\in\Gamma(M,V,\zeta)}a_{r,s}^{M}(\gamma)I_M(\gamma,f).\notag
\end{align}
\end{theorema}

\noindent The required identity will follow from the geometric and spectral expansions established in Theorems \ref{Igeom} and \ref{Ispec}, respectively. They are deduced inductively from the fine expansion of $J(f^r_s)$ above in a rather straightforward manner.  To prove the theorem, we first extend the results of \cite{FLM,FLss,FL} to the invariant linear form $I$, applied to the simpler test functions $f\times u^V$. This will take us part of the way in refining the coarse geometric expansion of $J(\dot f)$. {That is, we have the following (stated as Theorem \ref{invcont}).}

\begin{theorema}
The invariant linear form $I$ on $\H(G,V,\zeta)$ extends to a continuous linear form on $\C^\circ(G,V,\zeta)$. It has the spectral and geometric expansions given by
\begin{align*}
I(f)&= \sum_{M\in\L}|W^M_0||W^G_0|^{-1}\int_{\Pi(M,V,\zeta)}a^M(\pi)I_M(\pi,f)d\pi\\
&= \sum_{M\in\L}|W^M_0||W^G_0|^{-1}\sum_{\gamma\in\Gamma(M,V,\zeta)}a^M(\gamma)I_M(\gamma,f).
\end{align*}
\end{theorema}

\noindent The rest of the proof of Theorem \ref{theorem1}, which is the most difficult part, rests on transforming the global coefficients. The key difference here is that the global coefficients are now replaced with the new coefficients that depend on the basic function $b$, while the local distributions remain unchanged. In place of the unit element $u^V$, we are forced to work with the basic function $b^V$, a nontrivial and noncompactly supported function, at almost all places, thus complicating the necessary arguments. {The core of the analysis of the geometric side lies in the unipotent contribution, where we apply a variation of Arthur's linear independence argument, whereby the dependence on $b^V$ is subsumed into the global geometric coefficient $a^M_{r,s}(\gamma)$, leaving the usual weighted orbital integrals of $f$ at the finite places $V$. In particular, just as the automorphic $L$-function is a global object, we see that the role of the basic function in the trace formula is also {\em global} in nature, despite the fact that recent studies have focused on the local theory.}

\subsection{Outline} We conclude this introduction with a brief outline of the contents of the paper. In Section \ref{basic}, we introduce the necessary definitions and notation, and recall the properties of basic functions that we shall require. In Section \ref{conti} we prove the continuity of the trace formula for both the refined noninvariant trace formula and the invariant trace formula for $\dot f = f\times u^V$, which is a straightforward exercise in the definitions. Then in Sections \ref{invgeom} and \ref{invspec} we undertake the more serious work of establishing the geometric and spectral expansions of $I^r_s(f)$ respectively in \eqref{theorema} using the function $f^r_s = f\times b^V$, obtaining the new weighted global coefficients on each side, proving Theorem \ref{theorem1}.

\section{Basic notions}
\label{basic}

\subsection{Definitions}
\label{defs}

Let $G$ be a connected reductive group over a field $F$ of characteristic zero. We denote by $\L(M)$ to be the collection of Levi subgroups of $G$ containing $M$, $\L^0(M)$ the subset of proper Levi subgroups in $\L(M)$, and $\P(M)$ the collection of parabolic subgroups of $G$ containing $M$. Let $F$ be a global field, and $V$ a finite set of places of $F$. We have the real vector space $\a_M= \text{Hom}(X(M)_F,\R)$, and the set 
\[
\a_{M,V}=\{H_M(m) : m\in M(F_V)\}
\]
is a subgroup of $\a_M$, and $F_V = \prod_{v\in V} F_v$. It is equal to $\a_M$ if $V$ contains an archimedean place, and is a lattice in $\a_M$ otherwise. The additive character group $\a^*_{M,V} = \a^*_M\bs \a_{M,V}^\vee$ equals $\a^*_M$ in the first case, and is a compact quotient of $\a^*_M$ in the second. Let $A_M$ be the maximal split torus of a Levi subgroup $M$ of $G$. We then identify the Weyl group of $(G,A_M)$ with the quotient of the normaliser of $M$ by $M$, thus
\[
W^G(M)= \text{Norm}_G(M)/M.
\]
If $M_0$ is a minimal Levi subgroup of $G$, which we shall assume to be fixed, and denote $\L=\L(M_0), \P=\P(M_0), \L^0=\L^0(M_0),$ and $W^G_0=W^G(M_0)$. Also write $P_0 = M_0N_0$ for the minimal parabolic subgroup containing $M_0$.  

Let $Z$ be a central induced torus of $G$ over $F$. 
We define the pair $(Z,\zeta)$ where $\zeta$ is a character of $Z(F)$ if $F$ is local, and an automorphic character of $Z(\A)$ if $F$ is global. Given a finite set of places $V$, we write $G(F_v) = G(F_V)$ and write $\zeta_V$ for the restriction of $\zeta$ to the subgroup $Z(F_V)=Z(F_V)$ of $Z(\A)$. We then write $G(F_V)^Z=G(F_V)^Z$ for the set of  $x \in G(F_V)$ such that $H_G(x)$ lies in the image of the canonical map from $\a_Z$ to $\a_G$. We shall assume that $V$ contains the places over which $G$ and $\zeta$ are ramified. 

The stable trace formula requires that we work in fact with $G$ a $K$-group as defined in \cite[\S2]{ArtTW}. {Over $F_v$ nonarchimedean, it is again a connected reductive group, but over $F_v$ archimedean, it can be a finite union of connected reductive groups.} Thus
\[
G  = \coprod_{\alpha} G_\alpha\qquad \alpha\in \pi_0(G)
\]
is a variety whose connected components $G_\alpha$ are reductive groups over $F$, equipped with an equivalence class of frames
\[
(\psi,u) = \{(\psi_{\alpha\beta},u_{\alpha\beta}): \alpha,\beta\in\pi_0(G)\}
\]
satisfying natural compatibility conditions. Here $\psi_{\alpha\beta}: G_\alpha\to G_\beta$ in an isomorphism over $\bar{F}$, and $u_{\alpha\beta}$ is a locally constant function from $\Gamma = \text{Gal}(\bar{F}/F)$ to the simply connected cover $G_{\alpha,\text{sc}}$ of the derived group of $G_\alpha$. Any connected reductive group is a component of an $K$-group that is unique up to weak isomorphism. It comes with a local product structure 
\[
G(F_V) = \prod_{v\in V} \coprod_{\alpha_v\in\pi_0(G(F_v))}G_{\alpha_v}(F_v).
\]
The introduction of $K$-groups is to steamline certain aspects of {endoscopy over archimedean local fields}, and the definitions for connected groups will extend to $K$-groups in a natural way. For example, a central induced torus $Z$ of a $K$-group $G$ will have central embeddings $Z\stackrel{\sim}{\to}Z_\alpha\subset Z(G_\alpha)$ for each $\alpha$, and $\zeta$ determines a character $\zeta_\alpha$ for each $\alpha$. We shall call a $K$-group $G$ quasisplit if it has a connected component that is quasisplit over $F$.

Let $\C(G(\A_F)^1)$ be the space of Harish-Chandra Schwartz functions on $G(\A)^1$, and $\C(G(\A_F)^Z,\zeta)$ the $\zeta^{-1}$-equivariant functions on $G(\A)^Z$, {meaning that $f(zx) = \zeta^{-1}(z)f(x)$ for all $z\in Z(\A)$} \cite[\S9]{HCS}. We write $\C(G,V,\zeta) = \C(G(F_V)^Z,\zeta_V)$ for the space of $\zeta^{-1}$-equivariant Schwartz functions on $G(F_V)^Z$, which contains the $\zeta^{-1}$-equivariant Hecke algebra 
\be
\label{heck}
\H(G,V,\zeta) = \H(G(F_V)^Z,\zeta_V)
\ee
defined with respect to a choice of maximal compact subgroup $K_\infty$ of $G(F_{V_\infty})$, where $V_\infty$ denotes the archimedean places in $V$. If $F$ is a local field, we write $\C(G(F_v),\zeta_v)$ and $\H(G(F_v),\zeta_v)$ for the corresponding spaces. There are natural decompositions
\[
\C(G(F_v),\zeta_v) = \bigoplus_{\alpha_v\in\pi_0(G(F_v))} \C(G_{\alpha_v}(F_v),\zeta_{\alpha_v})
\]
and 
\[
\C(G(F_V),\zeta_V) = \bigotimes_{v\in V}\C(G(F_v),\zeta_v),
\]
and similarly for the Hecke algebra. 
We will also denote by $I\C(G(F_V)^Z,\zeta_V)$ and $S\C(G(F_V)^Z,\zeta_V)$ the spaces of orbital integrals and stable orbital integrals of functions in $\C(G(F_V)^Z,\zeta_V)$ respectively.
 
We also recall the space of functions constructed in \cite[\S3]{FLM} extending the usual space of test functions $C_c^\infty(G(\A_F)^1) = C_c^\infty(G(\A)^1)$. For any compact open subgroup $K$ of $G(\A_f)$ the space $G(\A)^1/K$ is a differentiable manifold. Any element $X\in \mathcal U(\mathfrak{g}^1)$, the universal enveloping algebra of the Lie algebra $\mathfrak{g}^1$ of $G(\R)^1= G(\R) \cap G(\A)^1$ defines a left-invariant differentiable operator $f*X$ on $G(\A)^1/K$. Let $\C^\circ(G(\A_F),K)$  be the space of smooth, right-$K$-invariant functions on $G(\A)^1$ which belong to $L^1(G(\A)^1)$ together with all their derivatives. It is a Fr\'echet space under the family of seminorms
\[
||f*X||_{L^1(G(\A)^1)}, \qquad X\in \mathcal U(\mathfrak{g}^1).
\]
For any nonnegative integer $k$, we define the norms
\[
||f ||_{G,k} = \sum_{i} ||X_i *f||_{L^1(G(\A)^1)},
\]
where $X_i$ ranges over a fixed basis of $\mathcal U(\mathfrak g)_{\le k}$ with respect to the standard filtration.
Denote by $\C^\circ(G(\A_F)^1)$ the union of $\C^\circ(G(\A_F),K)$ as $K$ varies over open compact subgroups of $G(\A_f)^1$, and endow $\C^\circ(G(\A_F)^1)$ with the inductive limit topology. 

As with the Hecke algebra, we shall also define the corresponding spaces $\C^\circ(G(\A_F)^Z,\zeta)$ and $\C^\circ(G,V,\zeta)$ obtained from the spaces of $\zeta^{-1}$-equivariant functions on $G(\A)^Z$ and $G(F_V)^Z$ respectively, in a manner parallel to $\C^\circ(G(\A_F)^1)$. The resulting spaces are natural subspaces of the Schwartz spaces $\C(G(\A_F)^1),\C(G(\A_F)^Z,\zeta)$, and $\C(G,V,\zeta)$ respectively. In particular, we shall apply results of Arthur for distributions on $\C(G(\A_F)^1)$ and so on to the smaller spaces $\C^\circ(G(\A_F)^1)$ and so on. Moreover, we will again take $G$ to be a $K$-group, so that 
\be
\label{Ccirc}
\C^\circ(G(\A_F)^1) = \bigoplus_{\alpha\in\pi_0(G)} \C^\circ(G_\alpha(\A_F))
\ee
and similarly with the spaces $\C^\circ(G(\A_F)^Z,\zeta)$ and $\C^\circ(G,V,\zeta)$. 

\subsection{Basic functions and local $L$-factors}

For the moment let $F$ be a nonarchimedean local field, and $G$ a reductive group defined over $F$. Suppose moreover that $G$ is unramified over $F$, meaning that $G$ admits a reductive model over $\O_F$. Recall that we have an exact sequence
\[
0 \to I_F \to \Gamma_F \to \Gamma_k \to 0
\]
where $I_F$ is the inertia group of $\Gamma_F$ , and $k$ is the residue field of $F$ with Frobenius element denoted by $\sigma_F$. If $G$ is quasisplit over $F$, it follows then that $G$ is unramified if and only if the restriction of the homomorphism $\Gamma_F \to \text{Out}(G\otimes_F\bar{F})$ to $I_F$ is trivial. In particular, the action of $\Gamma_F$ on $G^\vee$ factors through $\Gamma_k$, thus we may take $^LG$ to be $G^\vee \rtimes \langle \sigma_F\rangle$.

An irreducible smooth representation of $G(F)$ is unramified if it has a nonzero vector under $G(\O_F)$. Then the isomorphism classes of unramified representations $\pi$ of $G(F)$ are in canonical bijection with the conjugacy classes $\alpha$ of $G^\vee$ in the connected component of $^LG$, $\sigma_FG^\vee \subset G^\vee \rtimes \langle \sigma_F\rangle.$  Fix a maximal compact subgroup $K$ of $G(F)$. There is a twisted form of the Satake isomorphism
\be
\label{satake}
\mathrm{Sat}: \H(G(F),K)\to \CC[\sigma_FG^\vee]^{\mathrm{ad}(G^\vee)}
\ee
from the unramified Hecke algebra of $G(F)$ to the ring of regular functions on $\sigma_FG^\vee$ that are invariant under the adjoint action of $G^\vee$. The bijection $\pi \to \alpha_\pi$ is characterized by the requirement that $ \tr(\pi(f)) = \mathrm{Sat}(f)(\alpha_\pi)$ for any unramified irreducible representation $\pi$.


Given an irreducible complex finite-dimensional representation $r:{^LG}\to \GL(V)$, we have the local $L$-factor of $\pi$ given by
\[
L(s,\pi,r) = \det(1- r(\alpha_\pi)q^{-s})^{-1}
\]
where $q$ is the cardinality of the residue field  of $F$. We may expand it as a formal series
\[
\sum_{n=0}^\infty\tr((\mathrm{Sym}^nr)(\alpha_\pi))q^{-ns}
\]
converging absolutely for Re($s$) large enough, where the abscissa of convergence depends on the eigenvalues of $\alpha_\pi$. Viewing $\det(1- r(\alpha)q^{-s})^{-1}$ as a rational function on $G^\vee$, we would like to invert the Satake isomorphism to obtain a function $b^r_s$ such that 
\[
\tr(\pi(b^r_{s})) = L(s,\pi,r).
\]
Using the formal series expansion above, it is the same as asking for a family of functions $b^r_{n}$ in $\H(G(\A_F),K)$ such that
\be
\label{basicsum}
b^r_s = \sum_{n=0}^\infty b_n^r q^{-ns},
\ee
and $\tr(\pi({b}^r_{n})) = \tr(\text{Sym}^n(\alpha_\pi))$. 

The basic function $b^r_s$ is expected to be a distinguished vector in a certain Schwartz space of $r$-functions called the $r$-Schwartz space, defined as the global sections with compact support of a certain sheaf on a reductive monoid $M^r(F)$ containing $G(F)$ as an open subset \cite{takagi}. Following Ng\^o, we may assume that $G$ is equipped with a determinant homomorphism $\nu:G\to {\bf G}_m$ such that composition $r\circ \nu$ acts by scalar multiplication on the vector space $V$ of $r$, giving an exact sequence
\[
1 \to G_0 \to G \stackrel{\nu}{\to} {\bf G}_m\to 1
\]
where $G_0$ is a semisimple group. This is not a restrictive condition, seeing as we may replace $G$ by $G\times {\bf G}_m$ if necessary. Under this assumption, the sum \eqref{basicsum} is locally finite, and its support can be described explicitly by the weights of $r$. Basic functions have been the subject of much study of late, but as we shall see, our interest will lie not in the functions themselves but their orbital integrals. 

For our purposes, it will suffice to know that $b_s^r$ belongs to the spherical subspace $\H_\mathrm{ac}(G(F),K)$ of the almost-compact unramified Hecke space $\H_\mathrm{ac}(G(F))$ \cite[\S1]{ITF1}. Let $T$ be a maximal split torus of $G$ over $F$. Fix a Borel pair $(B,T)$ of $G$ and consider the Cartan decomposition $G(F) = KT(F)_+K$ using the anti-dominant Weyl chamber $X_*(T)_-$ in the cocharacter lattice $X_*(T)$, where $T(F)_+$ is the image of $X_*(T)_-$ under the map $\mu\mapsto \mu(\varpi)$ with $\varpi$ a uniformizer of $F$. The homomorphism $\nu$ induces a map $X_*(T) \to \Z$. In this case, Li has given an explicit description of the basic function
\be
\label{brs}
b^r_s = \sum_{\mu\in X_*(T)_-}c_\mu(q)\delta^\frac12_{B^-}(\mu(\varpi)){\bf1}_{K\mu(\varpi)K}q^{- \nu(\mu) s}
\ee
where $c(\mu)$ is polynomial in $q^{-1}$, and is a nonnegative integer given explicitly in terms of Kazhdan-Lusztig polynomials and symmetric powers of $r$. 

\begin{lem}
The basic function $b^r_s$ belongs to $\H_\mathrm{ac}(G(F),K)$ for $\mathrm{Re}(s)$ large enough. 
\end{lem}

\begin{proof}
The case where $G$ is split is due to  \cite[\S3]{Li}. If $G$ is quasisplit, we simply note that the Kato-Lusztig formula remains valid by \cite[Theorem 7.10]{Haines} and \cite[Theorem 1.9.1]{CCH}, and applying the Satake inversion \eqref{satake} for quasisplit $G$, it follows that the argument of \cite[Proposition 3.4]{Li} and the preceding discussion can be applied. 
\end{proof}

We shall identify $b^r_s$ with its $\zeta^{-1}$-equivariant analogue by replacing the characteristic functions ${\bf1}_{K\mu(\varpi)K}$ in \eqref{brs} with their $\zeta^{-1}$-equivariant analogues. It is also straightforward to extend the basic function to $K$-groups,
\[
b^r_s = \bigoplus_{\alpha\in\pi_0(G)}b^r_{\alpha,s}
\]
where $b^r_{\alpha,s}$ is the basic function defined by the component group $G_\alpha$, thereby placing us in proper generality.

\subsection{Weighting the trace formula}

We now return to $F$ being a global field. Enlarging $V$ if necessary, we shall assume that $G, \zeta$, and $r$ are unramified outside of $V$. Recall the set of families $C(G(\A^V),\zeta^V)$ of semisimple conjugacy classes in $^LG(F_v)$, for $v\not\in V$ in \cite[p.202]{STF1}. We shall in fact consider equivalence classes of families $c^V$, where two families $c$ and $c'$ in $C(G(\A^V),\zeta^V)$ are identified if $c_v = c'_v$ for almost all $v\not\in V$. Then given any $c$ in $C(G(\A^V),\zeta^V)$ and finite-dimensional representation $r$ of $^LG$, the Euler product 
\[
L^V(s,c,r) = \prod_{v\not\in V}\det(1- r(c_v))q_v^{-s})^{-1}
\]
converges to an analytic function in $s$ in some right-half plane. The local components $c_v$ determine unramified irreducible representations $\pi_v = \pi_v(c)$ of $G(F_v)$, and hence an unramified representation $\pi^V(c) = \otimes_{v\not\in V} \pi_v(c)$ of $G(\A^V)$. Then if $c$ is automorphic in the sense that there exists an irreducible representation $\pi_V$ of $G(F_V)$ such that $\pi = \pi_V \otimes \pi^V(c)$ is an automorphic representation of $G(\A)$, then a conjecture of Langlands asserts that $L^V(s,c,r)$ has meromorphic continuation \cite{Lprobs}. We can then identify the unramified automorphic $L$-function as 
\[
L^V(s,c,r) = L^V(s,\pi,r).
\]
If $\pi$ is tempered, the set of coefficients $\tr(r(c(\pi_v)))^k$ for $v\not\in V$ and $k\ge1$ is bounded; if moreover $\pi$ is cuspidal, one expects the $L$-function to have meromorphic continuation to the complex plane, with at most a simple pole at $s=1$. 

We can now properly describe the test functions that we shall use. Let $V_\text{ram}(G,\zeta)$ be the finite set of valuations of $F$ outside of which $G$ and $\zeta$ are unramified.  Fix a subset $V$ of $S$ containing $V_\text{ram}(G,\zeta)$. Given $f$ in $\C^\circ(G,V,\zeta)$, we shall form the test function 
\be
\label{frs}
\dot{f}^r_s = f \times b^V
\ee
in $\C^\circ(G(\A_F)^Z,\zeta)$, where
\[
b^V=b^V_{r,s} = \prod_{v\not\in V} b _{v,s}^r.
\]
so that 
\[
b^V_{G}(c) = b^V_{G}(\pi^V(c))= L^V(s,c ,r).
\]
More generally, for nonarchimedean valuations $v$ in $V$, we may choose  $f_v$ to be a $\zeta^{-1}$-equivariant function in $L^1(G(F_v))$ and let $f_\infty$ be a smooth $\zeta^{-1}$-equivariant function on $G(F_\infty)$ where $F_\infty = \prod_{v|\infty}F_v$ such that 
\[
||f_\infty * X||_{L^1(G(F_\infty))}<\infty, \qquad  X\in \mathcal U(\mathfrak{g}^1).
\]
It follows then that $\dot{f}^r_s$ belongs to $\C^\circ(G(\A_F)^Z,\zeta)$ for Re($s$) large enough.


\section{Continuity of the invariant trace formula}
\label{conti}

\subsection{The coarse expansion}

Let now $F$ be a number field, and let $\H(G(\A_F)^1)$ be the Hecke algebra on $G(\A)^1$. We first recall the noninvariant linear form $J(f)$ on $\H(G,V,\zeta)$ established in \cite[\S2]{STF1} from the original linear form on $\H(G(\A_F)^1)$.  It is a continuous, $Z(F)$-invariant linear form on $\C^\circ(G(\A_F)^1)$ consisting of two different expansions 
\[
J(\dot{f}^1) = \sum_{\o\in\O}J_\o(\dot{f}^1) = \sum_{\chi\in\mathfrak X }J_\chi(\dot{f}^1)
\]
for any $\dot{f}^1\in\H(G(\A_F)^1)$, with both sums converging absolutely. Here $\O$ is the set of $\O$-equivalence classes of element in $G(F)$, whereby two elements are equivalent if their semisimple parts are $G(F)$-conjugate, and $\mathfrak X$ is the set of equivalence classes of cuspidal automorphic data $\chi=\{(P,\sigma)\}$, where $P$ is a standard parabolic subgroup of $G$ with Levi subgroup $M_P$ and $\sigma$ is an irreducible representation of $M_P(\A)^1$, up to a certain equivalence relation as described in \cite{arthureis}. There is a natural projection 
\[
\dot{f}^1\to  \dot{f}^\zeta
\]
from $\C^\circ(G(\A_F)^1)$ onto the space $\C^\circ(G(\A_F)^Z,\zeta) = \C^\circ(G(\A)^Z,\zeta)$ given by
\be
\label{fzeta}
\dot{f}^\zeta(x) = \int_{Z(\A)^x}\dot{f}^1(zx) \zeta(zx) dz
\ee
where $x\in G(\A)^Z$ and $Z(\A)^x$ is the set of $z\in Z(\A)$ such that $H_G(zx)=0.$  We can then define a linear form on $\C^\circ(G(\A_F)^Z,\zeta)$ by
\be
\label{fzeta2}
J(\dot{f}^\zeta) = J^\zeta(\dot{f}^1) = \int_{Z(F)\bs Z(\A)^1} J(\dot{f}^1_z)\zeta(z) dz
\ee
where $\dot{f}^1_z$ denotes the translation of $\dot{f}^1$ by a point $z\in Z(\A)^1$, and the integral depends only on the image $\dot{f}^\zeta$ of $\dot{f}^1$ in $\C^\circ(G(\A_F)^Z,\zeta)$. Also, given a function $ f\in\C^\circ(G,V,\zeta)$, we can also define a linear form on $\C^\circ(G,V,\zeta)$ by setting
\[
J^r_s(f) = J(\dot{f}^r_s) 
\]
where $\dot{f}^r_s = f\times b^V$. 
We then have the noninvariant linear form on $\C^\circ(G,V,\zeta)$ given by
\[
J^r_s(f) = J(\dot{f}^r_s) = J^\zeta(\dot{f}^1)
\]
where $\dot{f}^1$ is any function in $\C^\circ(G(\A_F)^1)$ whose projection $\dot{f}^\zeta$ onto $\C^\circ(G(\A_F)^Z,\zeta)$ equals $\dot{f}^r_s=f\times b^V$. 

We next define an invariant linear form. $I^r_s$ on $\C^\circ(G,V,\zeta)$ inductively by setting
\be
\label{irs}
I^r_s(f) = J^r_s(f) - \sum_{M\in \L^0}|W^M_0||W^G_0|^{-1}\hat{I}^{r,M}_s(\phi_M(f))
\ee
for certain maps 
\be
\label{phi}
\phi_M: \H_\ac(G,V,\zeta)\to \I_\ac(M,V,\zeta)
\ee
constructed from normalized weighted characters \cite[(2.2)]{maps2} (see also \cite{canon}). To stabilize the invariant form $I^r_s$, we must first express the geometric and spectral expansions in terms of local distributions.

\subsection{The refined expansion}

For the remainder of this section,  we shall work more generally with the noninvariant linear form on $\H(G,V,\zeta)$ given by
\[
J(f) = J(\dot{f}) = J^\zeta(\dot{f}^1)
\]
where $\dot{f}^1$ is any function in $\H(G(\A_F)^1)$ whose projection $\dot{f}^\zeta$ onto $\H(G(\A_F)^Z,\zeta)$ equals $\dot{f}=f\times u^V$. It follows from the preceding discussion that $J(f)$ has the parallel expansions
\[
J(f) = \sum_{\o\in\O}J_\o(f) = \sum_{\chi\in\mathfrak X }J_\chi(f)
\]
which we would like extend to a larger family of noncompactly supported test functions. The following lemma extends the coarse expansion to a linear form on the space $\C^\circ(G,V,\zeta)$. 

\begin{lem}
\label{JC}
The linear form $J$ on $\H(G,V,\zeta)$ extends to a continuous linear form on $\C^\circ(G,V,\zeta)$.
\end{lem}

\begin{proof}
We follow the passage of $J$ from $\H(G(\A_F)^1)$ to $\H(G,V,\zeta)$. There is a natural projection
\[
\dot{f}^1\to \dot{f}^\zeta
\]
from $\C^\circ(G(\A_F)^1)$ to $\C^\circ(G(\A_F)^Z,\zeta)$ given by the formula \eqref{fzeta}. Given the linear form $J$ on $\C^\circ(G(\A_F)^1)$, we define a linear form on $\C^\circ(G(\A_F)^Z,\zeta)$ by
\[
J(\dot{f}^\zeta) = J^\zeta(\dot{f}^1) 
\]
where the right-hand side is defined as in \eqref{fzeta2}. 

Now let $f$ be a function in $\C^\circ(G,V,\zeta)$. Given any function $\dot{f}$ in $\C^\circ(G(\A_F)^1)$ whose projection $\dot{f}^\zeta$ onto $\C^\circ(G(\A_F)^Z,\zeta)$ equals $\dot{f}=f\times u^V$, we have the noninvariant linear form on $\C^\circ(G,V,\zeta)$ given by
\[
J(f) = J(\dot{f}) = J^\zeta(\dot{f}^1)
\]
as before, with both spectral and geometric sides converging absolutely. By the construction of the linear forms on each space, it follows that the form $J(f)$ on $\C^\circ(G,V,\zeta)$ is the continuous extension of the corresponding linear form on $\H(G,V,\zeta)$. 
\end{proof}

In order to pass to the invariant trace formula, we first have to refine the expansion of the noninvariant trace formula. In particular, we need to express both sides in terms of the basic distributions $J_M(\gamma,f)$ and $J_M(\pi,f)$. We first refine the geometric side. We refer to \cite[(2.8)]{STF1} for the construction of the global geometric coefficient $a^M(\gamma)$.

\begin{prop}
\label{geomprop}
Let $f\in \C^\circ(G,V,\zeta)$. Then the linear form $J(f)$ has a geometric expansion
\be
\label{Jgeom}
 \sum_{M\in\L}|W^M_0||W^G_0|^{-1}\sum_{\gamma\in\Gamma(M,V,\zeta)}a^M(\gamma)J_M(\gamma,f).
\ee
\end{prop}

\begin{proof}
The linear form $J(f)$ obtained in Lemma \ref{JC} has the coarse geometric expansion
\[
J(f) = \sum_{\o\in\O}J_\o(f) 
\]
with the sums converging absolutely. Let $G^0$ be the connected component of the identity in $G$, and $G_c$ the identity component of the centralizer of a semisimple element $c$ in $G(F)$. Then the equivalence class $\o$ consists of elements in $G(F)$ whose semisimple Jordan components belong in the same $G^0(F)$ orbit. There is another equivalence relation, which depends on a finite set of places $S$, which we shall assume contains $V$. The $(G,S)$-equivalence classes are defined to be the sets 
\[
G(F) \cap (\sigma U)^{G^0(F)} = \{g^{-1} \sigma u g: g\in G^0(F), u\in U\cap G^0(F)\}
\]
where $\sigma$ is a semisimple element of $G^0(F)$, and $U$ is a unipotent conjugacy class in $G_\sigma(F)$. Any class $\o\in\O$ breaks up into a finite set $(\o)_{G,S}$ of $(G,S)$-equivalence classes. 

Let $\dot{f}^1$ be any function in $\C^\circ(G(\A_F)^1)$ whose projection $\dot{f}^\zeta$ onto $\C^\circ(G(\A_F)^Z,\zeta)$ equals the function $\dot{f}=f\times u^V$. Suppose moreover that
\[
\dot{f}^1= \dot{f}^1_S \times u^{S,1},\qquad \dot{f}^1_S\in \C^\circ(G(F_S)^1).
\]
for $S\supset V$ large enough. The space $\C^\circ(G(F_S)^1)$ naturally embeds in $\C^\circ(G(\A_F)^1)$, while on the other hand any function in $\C^\circ(G(\A_F)^1)$ belongs to $\C^\circ(G(F_S)^1)$ for $S$ sufficiently large.  
It follows from \cite[Theorem 8.1]{orbits} that there is an expansion
\be
\label{ss}
J_\o(\dot{f}^1) = \sum_{M\in\L}|W_0^M||W_0^G|^{-1}\sum_{\dot\gamma\in(M(F)\cap \o)_{M,S}}a^M(S,\dot\gamma)J_M(\dot\gamma,\dot{f}^1_S)
\ee
for any $\o\in\O$, $\dot{f}^1_S\in C_c^\infty(G(F_S)^1),$ and $S$ containing a finite set $S_\o$ of valuations of $F$ including the archimedean places. Here $J_M(\dot\gamma,\dot{f}^1_S)$ is the weighted orbital integral of $\dot{f}^1_S$ over the conjugacy class of $\dot\gamma$ in $G(F_S)$, and is a tempered distribution by \cite{artFT}. The derivation of this formula relies on a combinatorial argument and descent to unipotent weighted orbital integrals, and in particular remains valid so long as the distribution $J_\o(\dot{f}^1)$ is absolutely convergent, and thus for $\dot{f}^1$ belonging to the larger space $\C^\circ(G(\A_F)^1)$. (We discuss the unipotent terms in greater detail in \cite[\S2]{maps2}.)

In order to sum over the classes $\o\in\O$, we have to modify the proof of \cite[Theorem 9.2]{orbits} and appeal to \cite[Theorem 7.1]{FL} instead for the convergence of the sum since $\dot{f}^1$ no longer has compact support. Let
\[
\text{ad}(G^0(\A))_\o=\{x^{-1}\gamma x: x\in G^0(\A),\gamma\in\o\},
\]
and write $\O_\Delta$ for the set of classes $\o$ such that $\text{ad}(G^0(\A))_\o$ meets the support of $\dot{f}^1_S$. Since $J_\o$ annihilates any function which vanishes on $\text{ad}(G^0(\A))_\o$, we obtain therefore
\[
\sum_{\o\in\O}J_\o(\dot{f}^1)  = \sum_{M\in\L}|W_0^M||W_0^G|^{-1}\sum_{\o\in\O_\Delta}\sum_{\dot\gamma\in(M(F)\cap \o)_{M,S}}a^M(S,\dot\gamma)J_M(\dot\gamma,\dot{f}^1_S).
\]
Now suppose that $\dot\gamma$ is any element of $(M(F))_{F,S}$. Then $\dot\gamma$ is contained in a unique class $\o\in\O$, and it follows from \cite[Theorem 5.2]{local} that $J_M(\dot\gamma,\dot{f}^1_S)$ vanishes if $\dot{f}^1_S$ vanishes on $\text{ad}(G^0(\A))_\o$, hence $J_M(\dot\gamma,\dot{f}^1_S)$ vanishes unless $\o$ belongs to $\O_\Delta$. From this we have that 
\[
J(\dot{f}^1) =  \sum_{M\in\L}|W^M_0||W^G_0|^{-1}\sum_{\dot\gamma\in (M(F_S))_{F,S}}a^M(S,\dot\gamma)J_M(\dot\gamma,\dot{f}^1_S).
\]
The rest of the argument is similar to the proof of \cite[Proposition 2.2]{STF1}, so we can be brief. For a fixed set of valuations $S$, the linear form $J(\dot{f}^1)$ is $K^S$-invariant, we may then write
\[
J(f) = \int_{Z(F)Z(\o^S)\bs Z(\A)^1}J(\dot{f}^1_z)\zeta(z)dz
\]
as 
\[
 \sum_{M\in\L}|W^M_0||W^G_0|^{-1}\sum_{\dot\gamma\in (M(F))_{M,S}}a^M(S,\dot\gamma)\int_{Z_{S,\o}\bs Z^1_S}J_M(z\dot\gamma,\dot{f}^1_S)\zeta(z)dz
\]
where
\[
Z_{S,\o}=Z(F)\cap Z(F_S)Z(\o^S)
\]
 and $\o^S = \prod_{v\not\in S}\o_v$, since $Z(\A) = Z(F)Z(F_S)Z(\o^S)$ and $J_M(\dot\gamma,\dot{f}^1_{S,z}) = J_M(z\dot\gamma,\dot{f}^1_S)$ for any $z\in Z(F_S)$. Then using the definition of the coefficient $a^M(\gamma)$, it follows that the geometric expansion of $J(f)$ can be written as 
\[
 \sum_{M\in\L}|W^M_0||W^G_0|^{-1}\sum_{\gamma\in\Gamma(M,V,\zeta)}a^M(\gamma)J_M(\gamma,f)
\]
as required.
\end{proof}

\begin{rem}We note that we have not obtained the absolute convergence of this refined geometric expansion. For semisimple elements $\gamma$, this follows from \cite[Theorem 1]{FLss}, which proves the absolute convergence of the semisimple contribution to \eqref{ss}, and by the argument above one  deduces the absolute convergence of the semisimple contribution to the refined geometric expansion \eqref{Jgeom}. As the authors point out, the absolute convergence of the unipotent contribution would require a uniform bound on the global geometric coefficients, which at present are known only for $\GL(n)$ \cite[Theorem 1.1]{matz}. Fortunately, this is not needed for the applications that we are interested in, which is the comparison of trace formulae. 
\end{rem}

We next refine the spectral expansion.

\begin{prop}
\label{Jspecprop}
Let $f\in \C^\circ(G,V,\zeta)$. Then the linear form $J(f)$ has a spectral expansion
\be
\label{Jspec}
 \sum_{M\in\L}|W^M_0||W^G_0|^{-1}\int_{\Pi(M,V,\zeta)}a^M(\pi)J_M(\pi,f)d\pi,
\ee
with the integrals converging absolutely.
\end{prop}

\begin{proof}
The linear form $J(f)$ obtained in Lemma \ref{JC} has the fine spectral expansion
\[
J(f) = \sum_{\chi\in\mathfrak X}J_\chi(f) 
\]
which converges absolutely, and where $J_\chi(f)$ is equal to the sum over $M\in\L$ of the product of 
\[
|W^M_0||W^G_0|^{-1}|\det(s-1)_{\a^G_M}|^{-1}
\]
with
\[
\sum_{\pi\in\Pi_\text{unit}(M,\zeta)}\sum_{L\in\L(M)}\sum_{s\in W^L(M)_\text{reg}}\int_{i\a^*_L/i\a^*_G}\tr(\J_L(P,\lambda)J_P(s,0)\I_{P,\chi,\pi}(\lambda,f))d\lambda,
\]
as stated in \cite[Theorem 8.2]{arthureis}. Here 
\[
\J_L(P,\lambda) = \lim_{\Lambda\to0}\sum_{Q\in\P(M)} \J_Q(P,\lambda,\Lambda)\theta_Q(\Lambda)^{-1},
\]
for $\Lambda\in i\a^*_M$ near to 0, is the limit of $(G,M)$-families
\[
\J_Q(P,\lambda,\Lambda) = J_{P|Q}(\lambda)^{-1}J_{Q|P}(\lambda+\Lambda)
\]
and $J_{Q|P}(\lambda)$ is the global unnormalized operator intertwining the actions of the induced representations $\I_P(\pi_\lambda)$ and $\I_Q(\pi_\lambda)$. Also
\[
J_P(s,0) = J_{P|P}(s,\pi_{\lambda+\Lambda}).
\]
It is a consequence of \cite[Corollary 1]{FLM} that the sums are finite and the integrals are absolutely convergent with respect to the trace norm, and define distributions on $\C^\circ(G(\A_F)^1)$. We note that the absolute convergence is proved for an expansion slightly different from the above, but is shown to be equivalent in \cite[\S5.3]{FLM}. Importantly, the sum over $\pi$ does not occur in the latter, but the necessary estimate for this sum, which is not necessarily finite, is contained in \cite[\S5.1]{FLM}.  (See also \cite[Theorem 7.2]{P} for the twisted case.)

Beginning with
\[
J(f) = J^\zeta(\dot{f}^1) = \int_{Z(F)\bs Z(\A)^1}J(\dot{f}^1_z)\zeta(z)dz,
\]
where $\dot{f}^1$ is any function in $\C^\circ(G(\A_F)^1)$ whose projection onto $\C^\circ(G(\A_F)^Z,\zeta)$ equals $\dot{f}=f\times u^V$, it follows from the argument of \cite[Theorem 4.4]{ITF2} and the definition of $a^G_\text{disc}(\dot\pi)$ that $J(f)$ has an expansion
\[
\int_{Z(F)\bs Z(\A)^1}\sum_{M\in\L}|W^M_0||W^G_0|^{-1}\sum_{\dot{\pi}\in\Pi_\text{disc}(M)}\int_{i\a^*_{M,Z}\bs i\a^*_{G,Z}} a^M_\text{disc}(\dot{\pi}_\lambda)J_M(\dot{\pi}_\lambda,\dot{f}^1_z)\zeta(z) d\lambda dz
\]
where
\[
J_M(\dot{\pi}_\lambda,\dot{f}^1_z) = \tr(\J_M(\dot{\pi}_\lambda,P)\I_P(\dot{\pi}_\lambda,\dot{f}^1_z))
\]
is the global unnormalized weighted character on $\C^\circ(G(\A_F)^1)$. It is a consequence of \cite[\S5.1]{FLM} that the inner integral converges absolutely. On the other hand, the integral over $Z(F)\bs Z(\A)^1$ annihilates the contribution of $\dot\pi$ coming from the complement of $\Pi_\text{disc}(M,\zeta)$ in $\Pi_\text{disc}(M)$, hence $J(f)$ equals 
\be
\label{jspec1}
\sum_{M\in\L}|W^M_0||W^G_0|^{-1}\sum_{\dot{\pi}\in\Pi_\text{disc}(M,\zeta)}\int_{i\a^*_{M,Z}/ i\a^*_{G,Z}} a^M_\text{disc}(\dot{\pi}_\lambda)J_M(\dot{\pi}_\lambda,\dot{f}) d\lambda.
\ee
Then arguing as in \cite[Proposition 3.3]{STF1}, it follows from the definition of $a^M(\pi)$ that the spectral expansion \eqref{jspec1} equals 
\[
\sum_{M\in\L}|W^M_0||W^G_0|^{-1}\int_{\Pi(M,V,\zeta)}a^M(\pi)J_M(\pi,f) d\pi
\]
where 
\[
J_M(\pi_\lambda,f) = \tr(\M_M(\pi_\lambda,P)\I_P(\pi_\lambda,f)),\qquad L\in \L(M),P\in\P(L)
\]
is the local normalized weighted character. It is related to the global unnormalized character by the formula
\[
J_M(\dot\pi_\lambda,\dot{f}) = \sum_{L\in\L(M)}r^L_M(c_\lambda) J_L(\pi^L_\lambda,f),
\]
and hence is defined for $f$ belonging to $\C^\circ(G,V,\zeta)$. Also, the operator $\J_Q(\Lambda,\dot{\pi}_\lambda,P)$ is a scalar multiple of $\M_Q(\Lambda,\pi_\lambda,P)$, that is,
\[
\J_Q(\Lambda,\dot{\pi}_\lambda,P) = r_Q(\Lambda,c_\lambda,P)\mu_Q(\Lambda,c_\lambda,P)\M_Q(\Lambda,\pi_\lambda,P),
\]
where the coefficient $r_Q(\Lambda,c_\lambda,P)$ is defined in \cite[\S2]{canon}, and it follows then that the integral over $\Pi(M,V,\zeta)$ converges absolutely. 
\end{proof}

\subsection{The invariant expansion}
Given the noninvariant linear form $J$ on $\H(G,V,\zeta)$, we have already discussed the invariant linear form $I$ also on $\H(G,V,\zeta)$ obtained by setting inductively
\be
\label{Idef}
I(f) = J(f) - \sum_{M\neq G}|W^M_0||W^G_0|^{-1}\hat{I}_M(\phi_M(f))
\ee
for the maps $\phi_M$ described in \eqref{phi}.

\begin{thm}
\label{invcont}
The invariant linear form $I$ on $\H(G,V,\zeta)$ extends to a continuous linear form on $\C^\circ(G,V,\zeta)$. It has the spectral and geometric expansions given by
\begin{align*}
&\sum_{M\in\L}|W^M_0||W^G_0|^{-1}\int_{\Pi(M,V,\zeta)}a^M(\pi)I_M(\pi,f)d\pi\\
&= \sum_{M\in\L}|W^M_0||W^G_0|^{-1}\sum_{\gamma\in\Gamma(M,V,\zeta)}a^M(\gamma)I_M(\gamma,f).
\end{align*}
\end{thm}

\begin{proof}
We recall that for any $\tilde{f}\in \C^\circ(G(F_V),\zeta_V)$, the function $\phi_M(\tilde{f})$ is defined to be the function on $\Pi_\text{temp}(M(F_V)^Z,\zeta_V)$ whose value at $\tilde{\pi}$ is the tempered distribution $J_M(\tilde{\pi},\tilde{f})$ \cite[\S2]{artFT}, and
\[
\phi_M(f,\pi) = \int_{i\a^*_{M,Z}}\phi_M(\tilde{f},\tilde{\pi})d\lambda
\]
where $f$ and $\pi$ are the restrictions of $\tilde{f}$ and $\tilde\pi$ to $G(F_V)^Z$ and $M(F_V)^Z$ respectively. We also define
\[
\phi_M(\tilde{f},\tilde\pi,X) = J_M(\tilde{f},\tilde{\pi},X),\qquad X\in \a_{M,V}
\]
and $\phi_M(f,\pi,X)$ using 
\[
J_M(\pi,X,f) = \int_{i\a^*_M}J_M(\pi_\lambda,f) \mathrm{e}^{-\lambda(X)}d\lambda
\]
if $J_M(\pi_\lambda,f)$ is regular for $\lambda\in i\a^*_M$. In this case, it follows from \cite[Lemma 3.1]{canon} that $\phi_M$ maps $\C^\circ(G(F_V)^Z,\zeta_V)$ continuously to $I\C^\circ(G(F_V)^Z,\zeta_V)$. For general $\pi$ in $\Pi(M,V,\zeta)$, it follows from the proof of Proposition \ref{Jspecprop} that $J_M(\pi,f)$ is well-defined for $f\in \C^\circ(G,V,\zeta)$, and moreover the integral
\[
J_M(\tilde{f},\tilde{\pi},X) = \int_{i\a^*_{M,V}/i\a^*_{G,V}}J_M(\tilde{\pi}_\lambda,\tilde{f}^Z)e^{-\lambda(X)}d\lambda
\]
converges absolutely. Here $Z$ is the image in $\a_{G,V}$ of $X$. 

On the other hand, the weighted orbital integrals $J_M(\gamma,f)$ are tempered distributions on $\C^\circ(G,V,\zeta)$ as a consequence of \cite[Theorem 4.1]{artFT}.
Altogether, it follows that the invariant distributions defined inductively by
\[
I_M(\pi,f) = J_M(\pi,f) -\sum_{L\in\L^0(M)}\hat{I}^L_M(\pi,\phi_L(f))
\]
and
\[
I_M(\gamma,f) = J_M(\gamma,f) -\sum_{L\in\L^0(M)}\hat{I}^L_M(\gamma,\phi_L(f))
\]
on either side of the invariant trace formula hold for functions $f$ in $\C^\circ(G,V,\zeta)$. 

Beginning with the linear form $J$ on $\C^\circ(G,V,\zeta)$, we define the invariant linear form $I$ as in \eqref{Idef}. We can see that the absolute value of $I(f)$ extends to a continuous linear form on $\C^\circ(G,V,\zeta)$, by assuming inductively that the statement holds for $L\in \L^0$ then applying the continuity of the map $\phi_M$ on $\C^\circ(G(F_V)^Z,\zeta_V)$ and the linear form $J$. But we shall also arrive at the same conclusion once we have obtained the desired expansions. Let us  first show that $I(f)$ has the geometric expansion
\[
I(f) =  \sum_{M\in\L}|W^M_0||W^G_0|^{-1}\sum_{\gamma\in\Gamma(M,V,\zeta)}a^M(\gamma)I_M(\gamma,f).
\]
Assume inductively that the required expansion holds if $G$ is replaced by any group $L\in\L^0$. Combining this with the geometric expansion \eqref{Jgeom} for $J$,
we see then that $I(f)$ equals 
\[
\sum_{M\in\L}|W^M_0||W^G_0|^{-1}\sum_{\gamma\in\Gamma(M,V,\zeta)}a^M(\gamma)\left(J_M(\gamma,f) - \sum_{L\in\L^0(M)}\hat{I}^L_M(\gamma,f)\right),
\]
and by definition of $I_M(\gamma,f)$ this yields the required geometric expansion for $I(f)$. On the other hand, the spectral expansion 
\[
I(f) =  \sum_{M\in\L}|W^M_0||W^G_0|^{-1}\int_{\Pi(M,V,\zeta)}a^M(\pi)I_M(\pi,f)d\pi
\]
follows in a similar manner. That is, assuming inductively that the required identity holds for $L\in\L^0$, and using the spectral expansion \eqref{Jspec} for $J$
it follows that $I(f)$ equals 
\[
\sum_{M\in\L}|W^M_0||W^G_0|^{-1}\int_{\Pi(M,V,\zeta)}a^M(\gamma)\left(J_M(\pi,f) - \sum_{L\in\L^0(M)}\hat{I}^L_M(\pi,f)\right)d\pi.
\]
Then by definition of $I_M(\pi,f)$ this yields the required spectral expansion for $I(f)$. 
\end{proof}

As we have alluded to in the beginning, the extension of the linear form $I$ to noncompactly-supported test functions in $\C^\circ(G,V,\zeta)$ does not yet allow for proper use of the basic function. To correct for this, we have to reconsider the passage from $\C^\circ(G(\A_F)^Z,\zeta)$ to $\C^\circ(G,V,\zeta)$, which requires, among other things, a reconsideration of the global geometric coefficients that depend on the finite set $S$ in a complicated way.

\section{Weighting the geometric side}
\label{invgeom}

The treatment of the geometric side is more involved. Let $\a_0 =\a_{M_0}$, and let $A_0$ be the split component of the center of $M_0$. The summands in the geometric expansion
\be
\label{Jrs}
J(\dot f^1) =\sum_{\o\in\O} J_\o(\dot f^1)
\ee
 are obtained by evaluating certain polynomials $J_\o^T(\dot{f}^1)$ at a distinguished point $T =  T_0$ in $\a_0$. We agree to write $J_\o(\dot{f}^1) = J_\o^{T_0}(\dot{f}^1).$ More precisely, let $\dot{f}^1\in \C^\circ(G(\A_F)^1)$ and $T$ be a point in the positive chamber $\a^+_0$ in $\a_0$ associated to $P_0$, suitably regular in the sense that its distance from the walls of $\a^+_0$ is large. Then $J_\o^T(\dot{f}^1)$ is the integral over $x\in G(F)\bs G(\A)^1$ of the function
\[
\sum_{P\in \P }(-1)^{\dim \a_P} \sum_{\delta\in P(F)\bs G(F)} k_{\o,P}(\delta x)\hat{\tau}_P(H_P(\delta x)-T)
\]
where $\hat{\tau}_P$ is the characteristic function of the set 
\[
\{ H\in \a_0 : \varpi(H) >0 ,\varpi \in \hat{\Delta}_P\},
\]
and 
\[
k_{\o,P}(\delta x) = \sum_{\substack{\gamma\in M_P(F)\\ I_P(\gamma) = \o}}\int_{N_P(\A)} \dot{f}^1(x^{-1}\gamma nx)dn.
\]
Here $\hat\Delta_P$ is the basis of $\a^*_P/\a_G^*$ which is dual to the simple roots $\Delta_P$ of $(P,A_P)$. As a function of $T$, $J^T_\o(\dot f^1)$ is a polynomial of degree at most $d_0 = \dim\a_0$, thus it can be extended to all $T\in\a_0$. Our present goal is to provide an expansion for \eqref{Jrs} as a distribution on $\C^\circ(G(\A_F)^1)$ in terms of local distributions. 
According to the proof of Proposition \ref{geomprop}, we can express the geometric side as
\[
J(\dot{f}^1) =  \sum_{M\in\L}|W^M_0||W^G_0|^{-1}\sum_{\dot\gamma\in (M(F))_{M,S}}a^M(\dot\gamma)J_M(\dot\gamma,\dot{f}^1_S),
\]
and from \cite[\S22]{Art}, it follows that the limit
\[
\lim_{S} J(\dot{f}^1)
\]
taken over increasing sets $S$, stabilizes for large finite $S$. Thus, in principle it may be possible to make use of the basic function in the form of $\dot{f}^r_s$ in the above limit, but we would like to have a more explicit form. For this, we shall revisit the refinement of the coarse geometric expansion.

\subsection{Unipotent terms}

We first have to deal with the unipotent contribution, which is the most delicate. It corresponds to the term
\[
J^T_\text{unip}(\dot{f}^1) = J^T_{\o}(\dot{f}^1)
\]
where $\o = \U_G(F)$, the Zariski closure in $G$ of the unipotent set in $G(F)$. It is a closed algebraic subvariety of $G$ defined over $F$, and is one of the classes in $\O$. We recall that the distribution $J_\text{unip}^T(\dot f)$ is obtained by integrating an alternating sum over standard parabolic subgroups, whose leading term is given by
\[
K_\text{unip}(x,x) = \sum_{\gamma\in \U_G(F)} \dot{f}^1(x^{-1}\gamma x).
\]
Let $(\U_G)$ be the set of Gal$(\bar{F}/F)$-orbits of $\U_G$. Then the previous expression can be rewritten as the sum over $U\in (\U_G)$ of
\[
K_U(x,x) = \sum_{\gamma\in U(F)}\dot{f}^1(x^{-1}\gamma x).
\]
In order to establish the refined geometric expansion for functions $ \dot f = f \times u^V$, where $f\in \C^\circ(G,V,\zeta)$ and $u^V$ is the $\zeta^{-1}$-equivariant characteristic function of the maximal compact subgroup $K^V$, we requier the existence of a measure on the unipotent variety for functions in $\C^\circ(G(\A_F)^1)$. We provide the argument here for $\dot{f}^r_s$, which we shall also need to construct our new geometric coefficients. 

Fix a Euclidean norm $||\cdot||$ on $\a_0$, and set $d(T) = \min_{\alpha\in\Delta_{P_0}}\{\alpha(T)\}$. Let $\Lambda^T_d$ be the Arthur's truncation operator applied to the diagonal \cite[p.1242]{unip}
 
\begin{lem}
\label{JTU}
There exist distributions $J^T_U$ for each $U\in \U_G$ which are polynomials in $T$ of total degree at most $d_0$ and such that
\[
J^T_\mathrm{unip}(\dot{f}^1) = \sum_U J^T_U(\dot{f}^1)
\]
for $\dot{f}^1 \in \C^\circ(G(\A_F)^1)$. Moreover, there is a continuous seminorm $\mu$ on $\C^\circ(G(\A_F)^1)$ and constants $\epsilon,\epsilon_0>0$ such that
\be
\label{JTUf}
\left|J^T_U(\dot{f}^1) - \int_{G(F)\bs G(\A)^1}\Lambda^T_d K_U(x,x) dx \right|\le \mu(\dot{f}^1) e^{\epsilon d(T)}
\ee
for all $U\in (\U_G),$  $\dot{f}^1\in \C^\circ(G(\A_F)^1)$ and every suitably regular $T$ with $d(T)\ge \epsilon_0||T||$.
\end{lem}

\begin{proof}
The proof of this statement is a mild generalization of \cite[Theorem 4.2]{unip}, where instead of the convergence estimate \cite[Theorem 3.1]{unip} for $\dot{f}^1\in C^\infty_c(G(\A_F)^1)$ we shall rely on \cite[Theorem 7.1]{FL} for $\dot{f}^1\in \C^\circ(G(\A_F)^1)$. 

Fix an orbit $U\in (\U_G)$. It is a locally closed subset of $G$, defined over $F$, and its Zariski closure $\bar U$ is a closed subvariety of $G$, again defined over $F$. The ideal of polynomial functions on $G$ which vanish on $U$ is of the form $(q_1,\dots,q_l)$, where $q_1,\dots,q_l$ are polynomials on $G$ defined over $F$. If $v$ is nonarchimedean valuation of $F$, we define $\rho_v$ to be the characteristic function of $[-1,1]$, and if $v$ is archimedean, define $\rho_v$ to be any function such that $0 \le \rho_v \le 1 $, equal to 1 on $[-\frac12,\frac12]$ and zero outside of $[-1,1]$.Then for any $\dot{f}^1\in \C^\circ(G(\A_F)^1)$ and $\epsilon>0$, we define the truncated function
\[
\dot{f}^{1,\epsilon}_{U,v}(x) = \dot{f}^1(x) \rho_v(\epsilon^{-1}|q_1(x)|_v)\cdots \rho_v(\epsilon^{-1}|q_l(x)|_v)
\]
where $x\in G(\A)^1$. It again belongs to $\C^\circ(G(\A_F)^1)$, and equals $f$ in a sufficiently small neighborhood of $\bar{U}(\A)$. 

Let $v$ be any valuation of $F$. We shall construct $J^T_U$ by examining the behavior of $J^T_\text{unip}(\dot{f}^{1,\epsilon}_{U,v})$ as $\epsilon$ approaches zero.  Let us write
\[
K_{\bar{U}}(x,x) = \sum_{\{U' \in(\U_G): U' \subset \bar{U}\}} K_{U'}(x,x).
\]
It will suffice to show there exists a continuous seminorm on $\C^\circ(G(\A_F)^1)$ such that for all $\dot{f}^1\in\C^\circ(G(\A_F)^1)$, the difference
\be
\label{un1}
\left|J^T_\text{unip}(\dot{f}^{1,\epsilon}_{U,v}) - \int_{G(F)\bs G(\A)^1}\Lambda_d^TK_{\bar{U}}(x,x)dx\right|
\ee
is bounded by
\be
\label{mubd}
\mu(\dot{f}^1)\delta^{rm}(1+||T||)^{d_0},
\ee
for some $\delta$ such that $0<\delta<1$, $r\ge0$, and $m$ large enough. The desired result will then follow by the same argument as in the proof of \cite[Theorem 4.2]{unip}. 

Given standard parabolic subgroups $P_1\subset P_2$, we write $A^\infty_{P_1}$ for the identity component $A_{P_1}(\R)^0$ of $A_{P_1}(\R)$, and $A^\infty_{P_1,P_2} = A_{P_1}^\infty\cap M_{P_2}(\A)^1$.  Moreover, given $T_1,T_2\in \a_0$, we denote by $A^\infty_{P_1,P_2}(T_1,T)$ the set
\[
\{a \in A^\infty_{P_1,P_2} : \alpha(H_{P_1}(a) - T_1) >0, \alpha\in\Delta^{P_1\cap M_{P_2}}; \varpi(H_{P_1}(a)-T)<0, \varpi\in \hat{\Delta}_{P_1\cap M_{P_2}}\}.
\]
Let $T\in \a_0$ be a suitably regular point. We define $F(x,T)$ to be the characteristic function of the compact subset of $G(F)\bs G(\A)^1$ obtained by the projection
\[
N_0(\A)M_0(\A) A^\infty_{P_0,G}(T_1,T) K \to G(F)\bs G(\A)^1.
\]
Using \cite[Lemma 2.3]{unip}, which states that
\[
\int_{G(F)\bs G(\A)^1} \Lambda^T_d K_U(x,x) dx = \int_{G(F)\bs G(\A)^1} \Lambda^T_d F(x,T)\left(\sum_{\gamma\in U(F)}\dot{f}^1(x^{-1}\gamma x) \right)dx 
\]
and the property that
\[
K_{\bar{U}}(x,x) = \sum_{\gamma\in \bar{U}(F)} f(x^{-1}\gamma x) = \sum_{\gamma\in \bar{U}(F)} \dot{f}^{1,\epsilon}_{U,v}(x^{-1}\gamma x),
\]
we may bound the difference \eqref{un1} by the sum of 
\be
\label{un2}
\left|J^T_\text{unip}(\dot{f}^{1,\epsilon}_{U,v}) - \int_{G(F)\bs G(\A)^1} F(x,T)\left(\sum_{\gamma\in \U_G(F)}\dot{f}^{1,\epsilon}_{U,v}(x^{-1}\gamma x) \right)dx \right|
\ee
and
\be
\label{un3}
\int_{G(F)\bs G(\A)^1}  F(x,T)\sum_{\gamma\in \U_G(F)\bs \bar{U}(F)}|\dot{f}^1(x^{-1}\gamma x)|dx.
\ee
The first expression \eqref{un2} is bounded by
\[
\mu(\dot{f}^{1,\epsilon}_{U,v})(1+ ||T||)^{d_0}e^{-d(T)}
\]
for some continuous seminorm $\mu_1$ on $\C^\circ(G(\A_F)^1)$, as an application of \cite[Theorem 7.1]{FL}. Replacing the seminorm $\mu(f)$ with $\mu(f) N(f)^n$ where the $N(f)$ is defined according to \cite[p.1257]{unip} and for $n$ large enough, we can apply \cite[Corollary 3.3]{unip}, which remains valid for functions in $\C^\circ(G(\A_F)^1)$, to conclude that the latter expression is bounded by
\[
\epsilon^{-l}\mu(\dot{f}^1) (1+ ||T||)^{d_0}e^{-d(T)}. 
\]
On the other hand, the second expression \eqref{un3} is bounded by
\[
\mu(\dot{f}^1) (1+||T||)^{d_0}\epsilon^r
\]
for some $r>0$, using \cite[Lemma 4.1]{unip}, which holds also in our case as the characteristic function $F(x,T)$ implies that the integral is taken over a compact set. Taking $\epsilon = \delta^m$ then, the required bound \eqref{mubd} follows. 
\end{proof}

We shall apply the lemma to obtain the following expansion for the unipotent term. 

\begin{prop}
\label{arss}
Fix a representation $r$ of $^LG$ and $s\in {\bf C}$ with $\mathrm{Re}(s)$ large enough. Then for any $S$, there are uniquely determined numbers 
\[
a^M_{r,s}(S,u), \qquad M\in \L, \ u \in (\U_M(F))_{M,S}
\]
such that 
\be
J^L_\mathrm{unip}(\dot{f}^1_{r,s}) = \sum_{M\in\L}|W^M_0||W^L_0|^{-1}\sum_{u\in(\U_M(F))_{M,S}} a^M_{r,s}(S,u)J_M(u,\dot{f}^1_S)
\ee
for any $L\in\L$ and $\dot{f}^1_{r,s} = \dot{f}^1_{S}\times b^{S,1}_{r,s}$ with $\dot{f}^1_S\in \C^\circ(G(F_S)^1)$.
\end{prop}

\begin{proof}
Assume inductively that the result holds for any Levi $M$ properly containing $L$. Define $T^L(\dot{f}^1_{r,s})$ to be the difference
\[
J^L_\text{unip}(\dot{f}^1_{r,s}) - \sum_{\substack{M\in \L^L\\ M\neq L}}|W^M_0||W^L_0|^{-1} \sum_{u\in(\U_M(F))_{M,S}} a^M_{r,s}(S,u)J^L_M(u,\dot f^1_S)
\]
for $\dot{f}$ as above. We can thus view $T^L$ as a distribution on $L_S^1$ that annihilates any function which vanishes on $\U_L(F_S)$. It is an invariant distribution by the same argument in pp.1269--1270 of \cite{unip}. We need to show that there exist uniquely determined numbers $a^L_{r,s}(S,u)$ such that 
\be
\label{TLf}
T^L(\dot{f}^1_{r,s}) =  \sum_{u\in(\U_L(F))_{L,S}} a^L_{r,s}(S,u)J^L_L(u,\dot f^1_S).
\ee
The uniqueness follows from the linear independence of the invariant orbital integrals $J^L_L(u)$, thus it remains to prove their existence.

For any integer $d$, let $\U_{L,d}$ be the union of orbits $U$ in $(\U_L)$ of maximal dimension $d$. The set 
\[
\U_{L,d}(F_S) = \prod_{v\in S} \U_{L,d}(F_v) 
\]
of $F_S$-valued points is a closed subspace of $L_S$ consisting of a finite union of $L_S$-conjugacy classes. Let $\U_{L,d}(F_S)'$ denote the union over orbits $U\in (\U_L)$ such that $\dim(U)\le d$ and such that $U(F)$ is nonempty, of the spaces $U(F_S)$. It is the union of $L_S$-conjugacy classes parametrized by elements $u\in (\U_L(F))_{L,S}$ with $\dim (U^L_u)\le d$. We see then that if there exist numbers 
\[
a^L_{r,s}(S,u), \qquad u\in (\U_L(F))_{L,S}
\]
such that for any $d$ the distribution 
\[
T^L_d(\dot{f}^1_{r,s}) = T^L(\dot{f}^1_{r,s}) - \sum_{\substack{u \in (\U_L(F))_{L,S}\\ \dim(U^L_u)>d}} a^L_{r,s}(S,u)J^L_L(u,\dot f^1_S) 
\]
annihilates any function $\dot f^1_S\in \C^\circ(G(F_S)^1)$ which vanishes on $\U_{L,d}(F_S)$, the required expression \eqref{TLf} will follow.

If $d \ge \dim(\U_L)$, then  $\U_{L,d}(F_S)'$ is the union of spaces $U_S$ such that $U(F)$ is not empty, and $T^L_d(\dot{f}^1_{r,s}) = T^L(\dot{f}^1_{r,s})$. In this case, $T^L_d(\dot{f}^1_{r,s})$ is the difference between the distribution obtained in Lemma \ref{JTU},
\[
J^L_\text{unip}(\dot{f}^1_{r,s}) = \sum_{U\in (\U_L)}J^L_U(\dot{f}^1_{r,s}),
\]
and a sum of integrals over $U(F_S)$ for which $U(F)$ is nonempty. Since $J^L_U$ is zero when $U(F)$ is empty, it follows that $T^L_d$ annihilates any function which vanishes on $\U_{L,d}(F_S)'$.

If $d > \dim (\U_L)$, assume inductively that $a^L_{r,s}(S,u)$ is defined for any $u$ with $\dim(U^L_u)>d$ and $^L_d$ annihilates any function which vanishes on $\U_{L,d}(F_S)'$. Let $\U_{L,d}^0$ be the union over orbits $U$ in $(\U_L)$ with $\dim(U)=d$, and let $C^d$ be the complement of $\U_{L,d}^0(F_S)$ in $\U_{L,d}(F_S)$. Thus $C^d$ equals to union over $v\in S$ and $U\in (\U_L)$ with $\dim(U)<d$ of the sets
\[
C^d_{U,v} = U(F_v) \prod_{\substack{w\in S\\ w\neq v}}\U_{L,d}(F_w).
\]
it is a closed subset of $L(F_S)^1$. We shall first consider the restriction of $T^L_d$ to the complement of $C^d$ in $L_S^1$. The space 
\[
\U_{L,d}(F_S)'\bs C^d = \U_{L,d}(F_S)'\cap \U_{L,d}^0(F_S)
\]
is a disjoint union of $L_S$-conjugacy classes which are closed in $L_S^1\bs C^d$. The conjugacy classes are parametrized by $u\in (\U_L(F))_{L,S}$ such that $\dim(U^L_u)=d$. For each such $u$, let $L_u$ be the centralizer in $L$ of a fixed representative of $u$ in $L(F)$. There is a surjective $L_S^1$-equivariant map
\[
\C^\circ(L_S^1)\bs C^d \to \bigoplus_u \C^\circ(L_S/L_{u,S}),
\]
with kernel consisting of functions in $\C^\circ(L_S^1\bs C^d)$ which vanish on $\U_{L,d}(F_S)'$. We may view any function that annihilates the kernel as the pullback of an $L_S^1$-equivariant distribution on the right-hand side. It follows then that we can choose constants $a^L_{r,s}(S,u)$ for each $u\in (\U_L(F))_{L,S}$ with $\dim(U^L_u)=d$ such that
\[
T^L_d(\dot{f}^1_{r,s}) = \sum_{\substack{u\in (\U_L(F))_{L,S}\\ \dim(U^L_u)=d}} a^L_{r,s}(S,u) J^L_L(\dot f^1_S)
\]
for any $\dot{f}^1_S\in \C^\circ(L_S^1\bs C^d)$. 

On the other hand, if $f$ is any arbitrary function in $\C^\circ(L_S^1)$, we set 
\[
T^L_{d-1}(\dot{f}^1_{r,s}) = T^L_d(\dot{f}^1_{r,s}) - \sum_{\substack{u\in (\U_L(F))_{L,S}\\ \dim(U^L_u)=d}} a^L_{r,s}(S,u) J^L_L(\dot f^1_S).
\]
Then $T^L_{d-1}(\dot{f}^1_{r,s})$ is an invariant distribution supported on $C^d$ which annihilates any function that vanishes on $\U_{L,d}(F_S)'$. By the inductive assumption, it will suffice to show that 
\[
T^L_{d-1}(\dot{f}^1_{r,s}) =0 
\] 
for any function $\dot{f}^1_{S}$ that vanishes on $\U_{L,d-1}(F_S)'$. Consider then the sets $C^d_{U,v}$ for which $\dot f^1_S$ does not vanish on a neighborhood of the closures $\bar{C}^d_{U,v}$. If no such set exists, then $\dot f^1_S$ belongs to $\C^\circ(L_S^1\bs C^d)$ and hence $T^L_{d-1}(\dot{f}^1_{r,s}) =0$. Otherwise, let there be $(k+1)$ such sets with $k\ge0$. Let $C^d_{U,v}$ be one such set. Then for any $\epsilon>0$, we have that $f^\epsilon_{U,v}$ is equal to $f$ in a neighborhood of $\bar{C}^d_{U,v}$, and $(f-f^\epsilon_{U,v})$ vanishes in a neighborhood of the closure of all but at most $k$ sets. We may assume inductively then that
\[
T^L_{d-1}(\dot{f}^1_{r,s}-(\dot{f}^1_{r,s})^\epsilon_{U,v}) = 0. 
\]
On the one hand, $T^L_{d-1}(\dot{f}^1_{r,s})$ is the difference between $J^L_\text{unip}(\dot{f}^1_{r,s})$ and 
\begin{align*}
&\sum_{\substack{M\subset L\\ M\neq L}} |W^M_0||W^L_0|^{-1} \sum_{u\in (\U_M(F))_{M,S}} a^M_{r,s}(S,u)J^L_M(u,\dot f^1_S)+ \sum_{\substack{u\in(\U_L(F)_{L,S}\\ \dim(U^L_u)\ge d}}a^L_{r,s}(S,u)J^L_L(u,\dot f^1_S).
\end{align*}
On the other hand, using the property that 
\[
\lim_{\epsilon\to0} J_M(u,(\dot{f}^1_{r,s})^\epsilon_{U,v}) = J_M(u,\dot{f}^1_{r,s}),
\]
if $U^G_u\subset \bar{U}$, otherwise it is zero, and  
\[
\lim_{\epsilon\to0} J^T_\text{unip}((\dot{f}^1_{r,s})^\epsilon_{U,v}) = \sum_{\{U' \in(\U_G): U' \subset \bar{U}\}}J^T_{U'}(\dot{f}^1_{r,s})
\]
for any valuation $v$, which follows from Lemma \ref{JTU} by the same argument as \cite[Corollary 4.3]{unip}, we then have that
\[
\lim_{\epsilon\to0} T_{d-1}^L((\dot{f}^1_{r,s})^\epsilon_{U,v}) 
\]
is equal to
\[
J^L_{\bar{U}}(\dot{f}^1_{r,s}) - \sum_{\substack{M\subset L\\M\neq L}} |W^M_0||W^L_0|^{-1}\sum_{\substack{u\in(\U_M(F))_{M,S}\\ U^L_u\subset \bar{U}}} a^M_{r,s,}(S,u)J^L_M(u,\dot f^1_S)
\]
for $\dim U<d$. Since $\dot{f}^1_{r,s}$ vanishes on $\U_{L,{d-1}}(F)'$, it follows from \eqref{JTUf} that $J^L_{\bar U}(\dot{f}^1_{r,s})=0$. The other terms in the preceding expression also vanish, thus $T^L_{d-1}(\dot{f}^1_{r,s})=0$ as desired. 
\end{proof}

Specializing to the case $L=G$, we have our desired expression for the unipotent contribution.

\begin{cor}
For any $\dot{f}^1_S\in \C^\circ(G(F_S)^1)$, we have 
\[
J_\mathrm{unip}(\dot{f}^1_{r,s}) =  \sum_{M\in\L}|W^M_0||W^G_0|^{-1}\sum_{u \in (\U_M(F))_{M,S}}a^M_{r,s}(S,\dot u )J_M(\dot u,\dot{f}^1_S).
\]
In particular, the restriction of $J_\mathrm{unip}$ to $G(F_S)^1$ is a measure.
\end{cor}

\subsection{Refinement}

Having treated the unipotent terms, the rest of the geometric expansion follows naturally. Let $M$ be a Levi subgroup of $G$, and $c$ a semisimple element of $M(F)$. For any $c\in G$, we denote by $G_{c,+}$ the centraliser of $c$ in $G$, and $G_c$ the connected component of the identity in $G_{c,+}$. We say a semisimple element $c$ is $F$-elliptic in $G$ if $A_{G_c}=A_G.$ We recall that two elements $\dot\gamma$ and $\dot\gamma_1$ in $G(F_S)$ with standard Jordan decompositions $\dot\gamma = c\dot\alpha$ and $\dot\gamma_1=c_1\dot\alpha_1$ are said to be $(G,S)$-equivalent if there is an element $\dot\delta\in G(F)$ such that $\dot\delta^{-1}c_1\delta = c$ and $\dot\delta^{-1}\alpha_1\delta$ is conjugate to $\alpha$ in $G_c(F_S)$. For a general element $\dot\gamma = c\dot\alpha$, we define the general coefficient by the descent formula
\be
\label{arsdesc}
a^G_{r,s}(S,\dot\gamma) = i^G(S,c)|\text{Stab}(c,\dot\alpha)|^{-1}a^{G_c}_{r,s}(S,\dot\alpha)
\ee
where $\text{Stab}(c,\dot\alpha)$ denotes the stabilizer of $\dot\alpha$ in the finite group $G_{c,+}(F)/G_c(F),$ and $i^G(S,c)$ is equal to 1 if $c$ is $F$-elliptic in $G$ and the $G(\A)^S$-conjugacy class of $c$ meets $K^S$, and is zero otherwise. 

\begin{prop}
\label{orbits}
Let $\dot{f}^1_S\in \C^\circ(G(F_S)^1)$. We then have
\[
J(\dot{f}^1_{r,s}) =  \sum_{M\in\L}|W^M_0||W^G_0|^{-1}\sum_{\dot\gamma\in (M(F))_{M,S}}a^M_{r,s}(S,\dot\gamma)J_M(\dot\gamma,\dot{f}^1_S).
\]
\end{prop}

\begin{proof}
The result will follow from the proof of \cite[Theorem 9.2]{orbits} if we can show that for each $\o\in \O$, there is a finite set $S_\o$ containing the archimedean places, such that for any finite set $S\supset S_\o$ and $f\in \C^\circ(G(F_S)^1)$,
\[
J_\o(f) =   \sum_{M\in\L}|W^M_0||W^G_0|^{-1}\sum_{\gamma\in (M(F)\cap\o)_{M,S}}a^M_{r,s}(S,\dot\gamma)J_M(\dot\gamma,\dot{f}^1_S)).
\]
This in turn follows from the proof of \cite[Theorem 8.1]{orbits}, where we need only indicate the changes that must be made in our setting. To that end, fix a semisimple element $c\in \o$ such that $c\in M_{P_1}(F)$ for a fixed standard parabolic $P_1$, and such that it does not belong to any proper parabolic subset of $M_1 = M_{P_1}$. Also let $\iota^G(c)=|G_{c,+}(F)/G_c(F)|$, and let $T$ be a suitably regular point in $\a_0$. Then following \cite[Lemma 3.1]{orbits}, the distribution $J_\o^T(f)$ can be expressed as the integral over $x$ in $G(F)\bs G(\A)^1$, and the sums over standard parabolic subgroups $R$ of $G_c$ and elements $\xi \in R(F)\bs G(F)$ of the product of
\[
|\iota^G(c)|^{-1}\sum_{u\in M_R(F)}\int_{N_R(\A)}f(x^{-1}\xi^{-1}c un\xi x) dn
\] 
with
\[
\sum_{P\in \F_R(M_1) }(-1)^{\dim (A_P/A_G)}\hat{\tau}_P(H_P(\xi x)-Z_P(T-T_0)-T_0).
\]
Here $\F_R(M_1)$ is the set of parabolic subsets $P$ with Levi factor $M_1$ with centralizer $P_c=R$, and $Z_P$ is defined in \cite[(3.3)]{orbits}.

Let $P_{1c} = P_1\cap G_c$, and let $M_{1c}$ be its Levi factor. Let $\F^c$ be the set of parabolic subgroups of $G_c$ containing $M_{1c}$. Let 
\[
\mathscr Y^T_R(\delta x,y) = \{Y^T_P(\delta x,y):P\in \F_R(M_1)\},
\]
where
\[
Y^T_P(\delta x,y) = -H_P(K_{P_c}(\delta x) y ) - Z_P(T-T_0)- T_c + T_0),
\]
and $K_{P_c}(x)$ is the component of $x$ in $K_c$ relative to the decomposition $G_c(\A) = N_{P_c}(\A)M_{P_c}(\A)K_c$. We would like to rewrite $J_\o^T(f)$ using a series of changes of variables as
\[
|\iota^G(c)|^{-1}\int_{G_c(\A)\bs G(\A)}\sum_{\{S\in \F^c: S \supset P_{1c}\}}\left(\int_{K_c}\int_{A^\infty_S\cap G(\A)^1}J_\text{unip}^{M_S,T_c}(\Phi^T_{S,a,k,y})da\ dk\right) dy
\]
where $\Phi^T_{S,a,k,y}(m), m \in M_S(\A)^1$ is given by
\[
\Gamma_S^G(H_S(a)- T_c, \mathscr Y^T_S(k,y))\delta_S(m)^\frac12\int_{N_S(\A)}f(y^{-1}c k^{-1}m n k y )dn
\]
and $\Gamma_S^G(X,\mathscr Y_R)$ is a compactly supported function on $X\in \a^G_R$, depending continuously on $\mathscr Y_R$ defined in \cite[\S4]{orbits}. We can do so by applying the combinatorial arguments of \cite[\S6]{orbits}, noting that the integral over $y$ remains absolutely integrable by the rapid decay of $f$, and that a weaker form \cite[Lemma 6.1]{orbits} holds by a similar argument (and simpler, as we do not require compactness). Namely, given a subset $\Delta$ of $G(\A)^1$ we can choose a subset $\Sigma$ of $G_c(\A)\bs G(\A)$ such that $y^{-1}c \U_{G_c}(\A) y \cap \Delta$ is empty unless $y$ belongs to $\Sigma$.  We can then choose $S_\o$ to be the finite set of valuations as described in \cite[p.203]{orbits} (see also \cite[p.193]{STF1}).
\end{proof}

We have the following formula for the global geometric coefficient in the case of semisimple elements. 

\begin{cor}
Let $\dot\gamma\in G(F)$ be a semisimple element. Then for any finite set $S\supset S_\o$, we have 
\[
a^G_{r,s}(S,\dot\gamma) = |G_{\dot\gamma,+}(F)/G_{\dot\gamma}(F)|^{-1} \vol(G_{\dot\gamma}\bs G_{\dot\gamma}(\A)^1)b^S_{r,s}(1)
\]
if $\dot\gamma$ is $F$-elliptic, and zero otherwise. 
\end{cor}

\begin{proof}
Let us first show that
\[
a^G_{r,s}(S,1) = a^G(S,1)b^S_{r,s}(1) = \mathrm{vol}(G(F)\bs G(\A)^1)b^S_{r,s}(1).
\]
Notice that if $U$ is the trivial unipotent class, then
\[
\int_{G(F)\bs G(\A)} \Lambda^T_dK_U(x,x)dx = \dot{f}^1_{r,s}(1) \int_{G(F)\bs G(\A)^1}F(x,T)dx,
\]
and by the dominated convergence theorem, we see that $J^T_{\{1\}}(\dot{f}^1_{r,s})$ is equal to the product of $\mathrm{vol}(G(F)\bs G(\A)^1)$ with $b^S_{r,s}(1)\dot{f}^1_S(1)$. From Proposition \ref{arss} we deduce the desired expression for trivial $\dot\gamma$, and the claim then follows from the descent formula \eqref{arsdesc}.
\end{proof}

\subsection{The invariant geometric expansion}

Following \cite[\S1]{STF1}, we now want to reindex the geometric terms in a different way.  Let $\D(G(F_V)^Z,\zeta_V)$ be the space of $\zeta_V$-equivariant distributions that are invariant under $G(F_V)^Z$-conjugation and supported on the preimage in $G(F_V)^Z$ of a finite union of conjugacy classes in $\overline{G}(F_V)^Z = G(F_V)^Z/Z(F_V)$. Let $\D_\text{orb}(G(F_V)^Z,\zeta_V)$ be the subspace spanned by distributions
\[
f \mapsto \int_{Z(F_V)}\zeta_V(z)f_G(z\gamma_V) dz, \qquad \gamma_V\in G(F_V)^Z
\]
where 
\[
f_G(\gamma_V)= |D(\gamma_V)|^{1/2}\int_{G_{\gamma_V}\cap G(F_V)^Z\bs G(F_V)^Z} f(x^{-1}\gamma_Vx) dx,
\]
for $G_{\gamma_V} = \prod_{v\in V} G_{\gamma_v}(F_v)$ and $|D(\gamma_V)|= \prod_{v\in V}|D(\gamma_v)|_v$ is the usual discriminant. Let $\Gamma(G(F_V)^Z,\zeta)$ be a fixed basis of $\D(G(F_V)^Z,\zeta_V)$, and let 
\[
\Gamma_\text{orb}(G(F_V)^Z,\zeta_V)= \Gamma(G(F_V)^Z,\zeta_V) \cap \D_\text{orb}(G(F_V)^Z,\zeta_V).
\]
Let $(\dot\gamma_S/\dot\gamma)$ be the ratio of the invariant measure on $\dot\gamma_S$ and the signed measure on $\dot\gamma_S$ that comes with $\dot\gamma$, so that 
\[
f_G(\dot\gamma_S) = (\dot\gamma_S/\dot\gamma) f_G(\dot\gamma),\qquad f\in \C^\circ(G(F_V)^Z,\zeta_V), \dot\gamma\in \Gamma_\text{orb}(\overline{G}(F_V)^Z,\zeta_V).
\]
For any $f\in \C^\circ(G,V,\zeta)$ and $\gamma\in \Gamma(M(F_V)^Z,\zeta_V)$, we inductively define the linear forms
\[
I_M(\gamma,f) = J_M(\gamma,f) -\sum_{L\in \L^0(M)}\hat{I}^L_M(\gamma,\phi_L(f))
\]
where $J_M(\gamma,f)$ is the generalized weighted orbital integral defined in \cite[\S2]{STF1} and $\phi_L$ is the map defined in \eqref{phi}.

We define  the elliptic coefficients
\be
\label{aell}
a^G_{r,s,\text{ell}}(\dot\gamma_S) = \sum_{\{\dot\gamma\}}|Z(F,\dot\gamma)|^{-1}a^G_{r,s}(S,\dot\gamma)(\dot\gamma/\dot\gamma_S)
\ee
where the sum over $\{\dot\gamma\}$ runs over a set of representatives of $Z_{S,\o}$-orbits in $(G(F))_{G,S}$, and $Z(F,\dot\gamma)$ is the subset of $z\in Z_{S,\o}$ such that $z\dot\gamma = \dot\gamma$. We note that $a^G_{r,s}(S,\dot\gamma)$ exists for any $S$-admissible element $\dot\gamma\in G(F)$, by \cite[Lemma 2.1]{STF1} and the analogue of \cite[Lemma 7.1]{orbits} for $f \in \C^\circ(G(F_S)^1)$ which follows from the proof of Proposition \ref{orbits}.

Let $V_\text{ram}(G,\zeta)$ be the finite set of valuations of $F$ outside of which $G$ and $\zeta$ are unramified.  We shall fix a subset $V$ of $S$ containing $V_\text{ram}(G,\zeta)$, and that $f$ is $S$-admissible in the sense of \cite[\S1]{STF1}. We specialize the test function $\dot{f}^r_s = f\times b^V_S$, hence
\[
f\to \dot{f}^r_s = f \times b^V_S
\]
gives a map from $\C^\circ(G,V,\zeta)$ to $\C^\circ(G,S,\zeta)$. We shall combine the elliptic coefficients with unramified weighted orbital integrals of basic functions at the places in $S-V$. Let $\K(\overline{G}(\A^V_S))$ denote the set of conjugacy classes in $\overline{G}(\A^V_S) = G(\A^V_S)/Z^V_S$ that are bounded in the sense that for any $v$ in $V-S$, the image of any representative lies in a compact subgroup of $G(F_v)$. Any element $k\in\K(\overline{G}(\A^V_S))$ induces a distribution $\gamma^V_S(k)$ in $\Gamma_\text{orb}(G(\A^V_S),\zeta^V_S)$. Given $\gamma$ in $\Gamma(G(F_V)^Z,\zeta_V)$, we write 
\[
\gamma\times k = \gamma\times \gamma^V_S(k)
\]
for the associated element in $\Gamma(G(F_S)^Z,\zeta_S)$. Furthermore, let $\K^V_\text{ell}(\bar{M},S)$ denote the elements in $\K(\overline{G}(\A^V_S))$ such that $\gamma\times k$ belongs to $\Gamma_\text{ell}(G,S,\zeta)$ for some $\gamma$. Here $\Gamma_\text{ell}(G,S,\zeta)$ is the set of $\gamma$ in $\Gamma_\text{orb}(G(F_S)^Z,\zeta_S)$ such that there is a $\dot\gamma\in G(F)$ such that
\ben
\item[(i)] the semisimple part of $\dot\gamma$ is $F$-elliptic in $G$, 
\item [(ii)] the conjugacy class of $\dot\gamma$ in $G(F_v)$ maps to $\gamma$, and 
\item[(iii)] $\dot\gamma$ is bounded at each $v\not\in S$.
\een
 We can then define the unramified weighted orbital integrals
\be
\label{rkb}
r^G_M(k,b) = J_M(\gamma^V_S(k),b^V_S), \qquad k\in \K(\bar{M}(\A^V_S)).
\ee
Now the set $\Gamma(G,V,\zeta)$ is given by the union of induced distributions $\mu^G$ where $\mu$ runs over elements in $\Gamma_\text{ell}(M,V,\zeta)$ and $M$ runs over Levis in $\L$. Recall that the induction is defined by the relation 
\[
f_G(\mu^G) = f_M(\mu),\qquad f\in \C^\circ(G(F_V),\zeta_V), 
\]
and $f\mapsto f_M$ is the canonical map from $\C^\circ(G(F_V),\zeta_V)$ to $I\C^\circ(G(F_V),\zeta_V)$ factoring through the map $f\mapsto f_G$. We also have the adjoint restriction map $\gamma\mapsto \gamma_M$ from $\D(G(F_V),\zeta_V)$ to $\D(M(F_V),\zeta_V)$, such that
\be
\label{muG}
\sum_{\gamma\in\Gamma(G(F_V),\zeta_V)}a_M(\gamma_M)b_G(\gamma) = \sum_{\mu\in\Gamma(M(F_V),\zeta_V)}a_M(\mu)b_G(\mu^G)
\ee
for any linear function $a_M$ on $\D(M(F_V),\zeta_V)$ and $b_G$ on $\D(G(F_V),\zeta_V)$. We then define for any $\gamma\in \Gamma(G(F_V)^Z,\zeta_V)$, the geometric coefficient
\be
\label{ars}
a^G_{r,s}(\gamma) =\sum_{M\in\L}|W^M_0||W^G_0|^{-1}\sum_{k\in \K^V_\text{ell}(\bar{M},S)}a^M_{r,s,\text{ell}}(\gamma_M\times k)r^G_M(k,b)
\ee
where $S$ is any finite set of valuations of containing $V$ such that $\gamma\times K^V$ is $S$-admissible. It follows from the definitions that $a^G_{r,s}(\gamma)$ is supported on the discrete subset $\Gamma(G,V,\zeta)$ of $\Gamma(G(F_V)^Z,\zeta_V)$.

\begin{thm}
\label{Igeom}
Let $f\in \C^\circ(G,V,\zeta)$. Then the invariant linear form $I^r_s(f)$ has the geometric expansion
\be
\label{Irsg}
I^r_s(f) = \sum_{M\in\L}|W^M_0||W^G_0|^{-1}\sum_{\gamma\in\Gamma(M,V,\zeta)} a^M_{r,s}(\gamma)I_M(\gamma,f)
\ee
\end{thm}

\begin{proof}
Recall from Proposition \ref{orbits} the expression
\[
J(f^r_s) = J(\dot{f}_{r,s}^1) =  \sum_{M\in\L}|W^M_0||W^G_0|^{-1}\sum_{\dot\gamma\in (M(F_S))_{F,S}}a^M_{r,s}(S,\dot\gamma)J_M(\dot\gamma,\dot{f}^1_S).
\]
For a fixed set of valuations $S$, the linear form $J(\dot{f}^1)$ is $K^S$-invariant, we may then write
\[
J(f^r_s) = \int_{Z(F)Z(\o^S)\bs Z(\A)^1}J(\dot{f}^1_{r,s,z})\zeta(z)dz
\]
as 
\[
 \sum_{M\in\L}|W^M_0||W^G_0|^{-1}\sum_{\dot\gamma\in (M(F))_{M,S}}a^M_{r,s}(S,\dot\gamma)\int_{Z_{S,\o}\bs Z^1_S}J_M(z\dot\gamma,\dot{f}^1_S)\zeta(z)dz
\]
since $Z(\A) = Z(F)Z(F_S)Z(\o^S)$ and $J_M(\dot\gamma,\dot{f}^1_{S,z}) = J_M(z\dot\gamma,\dot{f}^1_S)$ for any $z\in Z(F_S)$. The inner sum can be written as
\[
\sum_{\{\dot\gamma\}}|Z(F,\dot\gamma)|^{-1}a^M_{r,s}(S,\dot\gamma)\int_{Z^1_S}J_M(z\dot\gamma,\dot{f}^1_S)\zeta(z)dz
\]
Then applying the definition of \eqref{aell}, we have
\be
\label{jrs}
J^r_s(f) = \sum_{M\in\L}|W^M_0||W^G_0|^{-1}\sum_{\dot\gamma_S\in\Gamma_\text{ell}(M,V,\zeta)} a^M_{r,s,\text{ell}}(\dot\gamma_S)J_M(\dot\gamma,\dot f_S)
\ee
where $\dot f_S$ is the projection of $\dot{f}_S^1$ onto $\H(G,S,\zeta)$.

Since $\dot{f}_S$ is equal to $f \times b^V_S$, we claim that
\[
J_M(\dot\gamma_S,\dot{f}_S) = \sum_{L\in \L(M)}J^L_M(\dot\gamma^V_S,(b^V_S)_L)J_L(\dot\gamma^L_V,f).
\]
This follows from the descent and splitting properties of the weighted orbital integrals, stated in (18.7) and (18.8) of \cite{Art} or, more directly in \cite[VI.1.9(2)]{MW2}). Here we use the fact that $(b^V_S)_L=(b^V_S)_Q$ is independent of $Q\in \P(L)$, where
\[
f_Q(m) = \delta_Q(m)^\frac12\int_K \int_{N_Q(F_S)}f(k^{-1} mnk) dn\ dk,
\]
since by \eqref{brs} we see that the basic function depends only on the restriction to $T(F_v)_+$ for each $v\in V - S$. Moreover, $J^L_M(\dot\gamma^V_S,(b^V_S)_L)$ vanishes unless $\dot\gamma^V_S = \dot\gamma^V_S(k)$ for some $k\in \K(\bar{M}(\A^V_S))$, in which case it equal $r^L_M(k,b)$ by definition. Hence $a^M_{r,s,\text{ell}}(\dot\gamma_S)$ vanishes unless $\mu = \dot\gamma_V$ lies in $\Gamma_\text{ell}(M,V,\zeta)$ and $k$ lies in $\K^V_\text{ell}(\bar{M},S)$. We can thus write \eqref{jrs} as
\[
\sum_{L\in\L}\sum_{M\in\L^L}|W^M_0||W^G_0|^{-1}\sum_{\mu\in\Gamma_\text{ell}(M,V,\zeta)}\sum_{k\in \K^V_\text{ell}(\bar{M},S)} a^M_{r,s,\text{ell}}(\mu\times k)r^L_M(k,b)J_L(\mu^L,f).
\]
Using the property \eqref{muG} and the definition \eqref{ars} it follows that the inner sum can be expressed as
\begin{align*}
&\sum_{M\in\L}|W^M_0||W^G_0|^{-1}\sum_{\gamma\in\Gamma(L,V,\zeta)}\sum_{k\in\K^V_\text{ell}(\bar{M},S)} a^M_{r,s,\text{ell}}(\gamma_M\times k)r^L_M(k,b)J_L(\gamma,f)\\
&=|W^M_0||W^G_0|^{-1}\sum_{\gamma\in\Gamma(L,V,\zeta)}a^L_{r,s}(\gamma) J_L(\gamma,f).
\end{align*}
Writing $M$ for $L$ in the preceding expression, the geometric expansion of $J^r_s(f)$ can thus be written as 
\[
J^r_s(f) =  \sum_{M\in\L}|W^M_0||W^G_0|^{-1}\sum_{\gamma\in\Gamma(M,V,\zeta)}a^M_{r,s}(\gamma)J_M(\gamma,f).
\]

Converting this expansion for $J(f^r_s)$ into an expansion for $I(f^r_s)$ is standard. Recall from the definition in \eqref{irs} that
\[
I^r_s(f) = J^r_s(f) - \sum_{M\in \L^0}|W^M_0||W^G_0|^{-1}\hat{I}^{r,M}_s(\phi_M(f)).
\]
Assume inductively that the expansion \eqref{Irsg} holds if $G$ is replaced by any proper Levi $L\in \L^0$. Then using the expansion we have just obtained for $J^r_s(f)$, we see that
\[
I^r_s(f) =  \sum_{M\in\L}|W^M_0||W^G_0|^{-1}\sum_{\gamma\in\Gamma(M,V,\zeta)}a^M_{r,s}(\gamma)\left(J_M(\gamma,f) - \sum_{L\in \L(M)}\hat{I}^L_M(\gamma,\phi_L(f))\right).
\]
By the definition of $I_M(\gamma,f)$ then this is equal to 
\[
\sum_{M\in\L}|W^M_0||W^G_0|^{-1}\sum_{\gamma\in\Gamma(M,V,\zeta)}a^M_{r,s}(\gamma)I_M(\gamma,f)
\]
as required.
\end{proof}

\begin{cor}
The coefficients $a^G_{r,s}(\gamma)$ are independent of $S$.
\end{cor}

\begin{proof}
The linear form $I^r_s(f)$ is constructed from the noninvariant form $J^r_s(f)$, which is independent of $S$. Assume inductively that for any proper Levi subgroup $M$ of $G$, the coefficients $a^M_{r,s}(\gamma)$ are independent of $S$. Then the terms corresponding to the $M$ in \eqref{Irsg} are independent of $S$, thus so is the term corresponding to $G$, which is
\[
\sum_{\gamma\in\Gamma(M,V,\zeta)} a^G_{r,s}(\gamma)f_G(\gamma).
\]
It follows then that the $a^G_{r,s}(\gamma)$ are independent of $S$.
\end{proof}

\section{Weighting the spectral side}
\label{invspec}

On the spectral side, we have the sum
\[
J(\dot{f}^1) = \sum_{\chi\in\mathfrak X} J_\chi(\dot{f}^1),
\]
whose summands are also obtained by evaluating the polynomials $J_\chi^{T}(\dot{f}^1)$ at a distinguished point $T\in \a_0$. They are simpler to treat than the geometric expansion, given the absolute convergence of the spectral side in hand, together with the results of \cite{FLM} which concern the refined noninvariant spectral expansion. 

\subsection{Refinement} Let $\Pi_\text{unit}(M(\A)^1)$ be the set of unitary representations of $M(\A)^1$. Given an element $\dot\pi\in \Pi_\text{unit}(M(\A)^1)$ and $\lambda\in i\a^*_M$, we form the representation $\pi_\lambda$ by multiplying by $e^{\lambda (H_P(\cdot))}$ and the induced representation $\I_P(\dot\pi_\lambda)$ for any $P\in\P(M)$. Let $W^L(M)_\text{reg}$ be the subset of regular elements in $W^L(M)$, {that is, elements in $W^L(M)$ with kernel equal to $\a_L$.} Also let $\a^*_{G,Z}$ be the subspace of linear forms on $\a_G$ that are trivial on the image of $\a_Z$ in $\a_G$. For any $Q\in \P(M)$ and $s\in W(M)$ there is a global (unnormalized) unitary intertwining operator $J_{Q|P}(s,\dot\pi_\lambda)$ from $\H_P(\dot\pi)$ to $\H_Q(\dot\pi).$ We set \[
J_{Q|P}(1,\dot\pi_\lambda) = J_{Q|P}(\dot\pi_\lambda),\qquad J_P(s,0) = J_{P|P}(s,\dot\pi_{\lambda+\Lambda}).
\]
We recall that there is a discrete subset $\Pi_\disc(G(\A_F)^1)$ that supports a finite linear combination of characters
\[
I_\disc(\dot{f}^1) = \sum_{\dot\pi\in \Pi_\disc(G(\A_F)^1)}a^G_\disc(\dot\pi) \dot f^1_G(\dot\pi), 
\]
with $\dot f^1\in\C^\circ(G(\A_F)^1)$. The spectral coefficients $a^G_\disc(\dot\pi) $ are the multiplicities defined by the spectral expansion of the discrete part 
\[
I_\disc(\dot f^1) = \sum_{M\in\L}|W^M_0||W^G_0|^{-1}\sum_{s\in W^L(M)_\text{reg}}|\det(s-1)_{\a^G_M}|^{-1}\tr(J_P(s,0)\I_P(\dot\pi,\dot{f}^1))
\]
in \cite[\S4]{ITF2}, written as a linear combination of characters. 

We begin by recording the following spectral expansion of $J(\dot{f}^1_{r,s})$. 

\begin{prop}
\label{eis}
Let $\dot{f}^1_V\in \C^\circ(G(F_V)^1)$. We then have
\[
J(\dot{f}^1_{r,s}) =  \sum_{M\in\L}|W^M_0||W^G_0|^{-1}\sum_{\dot\pi\in \Pi_\mathrm{disc}(M)}a^M_\mathrm{disc}(\dot\pi_\lambda)\int_{i\a_{M,Z}^*/i\a_{G,Z}^*}J_M(\dot\pi_\lambda,\dot{f}^1_{r,s})d\lambda.
\]
\end{prop}
\begin{proof}
We recall that $J_\chi(\dot{f}^1)$ is equal to the sum over $M\in\L$ of the product of 
\[
|W^M_0||W^G_0|^{-1}|\det(s-1)_{\a^G_M}|^{-1}
\]
with
\[
\sum_{\pi\in\Pi_\text{unit}(M(\A)^1)}\sum_{L\in\L(M)}\sum_{s\in W^L(M)_\text{reg}}\int_{i\a^*_L/i\a^*_G}\tr(\J_L(\dot\pi_\lambda,P)J_P(s,0)\I_{P}(\dot\pi_\lambda,\dot{f}))d\lambda,
\]
as in the proof of Proposition \ref{Jspecprop}. Here 
\[
\J_L(\dot\pi_\lambda,P) = \lim_{\Lambda\to0}\sum_{Q\in\P(M)} \J_Q(\Lambda,\dot{\pi}_\lambda,P)\theta_Q(\Lambda)^{-1},
\]
for $\Lambda\in i\a^*_M$ near to 0, is the limit of $(G,M)$-families
\[
\J_Q(\Lambda,\dot{\pi}_\lambda,P) = J_{P|Q}(\dot{\pi}_\lambda)^{-1}J_{Q|P}(\dot{\pi}_{\lambda+\Lambda}).
\]
Then the required formula follows by the same argument in the proof of \cite[Theorem 4.4]{ITF2}.
\end{proof}

\subsection{The invariant spectral expansion}

As before, we have to convert this expansion to a distribution on $G(F_V)^Z$. We continue to assume that $V$ contains $V_\text{ram}(G,\zeta)$. Let $\F(G(F_V)^Z,\zeta_V)$ be the space of finite complex linear combinations of irreducible characters on $G(F_V)^Z$ with $Z(F_V)$-central character equal to $\zeta_V$, with a canonical basis $\Pi(G(F_V)^Z,\zeta_V)$ of irreducible characters. We identify elements $\pi\in \F(G(F_V)^Z,\zeta_V)$ with the linear form 
\[
f\mapsto f_G(\pi) =\tr(\pi(f))
\]
on $\H(G(F_V)^Z,\zeta_V)$. We write $\Pi_\text{unit}(G(F_V)^Z,\zeta_V)$ for the subset of unitary characters. The orbits of the action 
\[
\pi\mapsto \pi_\lambda,\qquad \lambda\in i\a^*_{G,Z}
\]
 on $\pi\in \Pi_\text{unit}(G(F_V),\zeta_V)$ can be identified with the set $\Pi_\text{unit}(G(F_V)^Z,\zeta_V)$. We also define the induced characters 
\[
f_M(\pi)= f_G(\pi^G)=\tr(\I_P(\pi,f)),\qquad \pi\in \Pi_\text{unit}(M(F_V),\zeta_V).
\]
 Given any $f\in\H(G,V,\zeta)$ and $\pi\in \Pi_\text{unit}$, we inductively define the linear form
\[
I_M(\pi,f) = J_M(\pi,f) - \sum_{M\in\L^0(M)}\hat{I}^L_M(\pi,\phi_L(f))
\]
where $J_M(\pi,f)$ is the weighted character defined in \cite[\S3]{STF1} and $\phi_L$ is the map defined in \eqref{phi}.

We shall combine the spectral coefficients with unramified characters using basic functions at places outside of $V$. Let us define $\Pi_\disc(G(\A_F)^Z,\zeta)$ to be the set of representations in $\Pi(G(\A)^Z,\zeta)$ whose restrictions to $G(\A)^1$ lie in $\Pi_\disc(G(\A_F)^1)$. It can be identified with the representations in $\Pi_\disc(G(\A_F)^1)$ whose central character on $Z(\A)^1$ is equal to $\zeta$.  We then define $\Pi_\disc(G,V,\zeta)$ to be the subset of $\pi \in \Pi(G(F_V)^Z,\zeta_V)$ such that $\pi\times c$ belongs to $\Pi_\disc(G(\A_F)^Z,\zeta)$ for some $c\in C(G(\A^V),\zeta^V)$. Now given $c\in C(G(\A^V),\zeta^V)$, there is a natural action of $\lambda \in \a^*_{G,Z,{\bf C}}$ sending 
\[
c\to c_\lambda = \{c_{v,\lambda}:v\not\in V\}. 
\]
Let $\pi\times c = \pi\otimes \pi^V(c)$ denote the associated representation in $\Pi(G(\A),\zeta)$, and let $C_\text{disc}(G(\A_F)^Z,\zeta)$ denote the elements $c\in C(G(\A^V),\zeta^V)$ such that $\pi\times c$ belongs to $\Pi_\text{disc}(G(\A_F)^Z,\zeta)$ for some $\pi\in \Pi_\disc(G,V,\zeta)$.  Following Arthur, we define for any $c\in C^V_\disc(G(\A_F)^Z,\zeta)$ the unramified normalizing factors $r_{Q|P}(c_\lambda)$ as the quotient of completed automorphic $L$-functions  
\[
L(0,c_\lambda,\rho_{Q|P})L(1,c_\lambda,\rho_{Q|P})^{-1},
\]
for $P,Q\in \P(M)$, where $\rho_{Q|P}$ is the adjoint representation of $^LM$ on the Lie algebra of the intersection of the unipotent radicals of $\hat{P}$ and $\hat{Q}$. We also form the corresponding $(G,M)$-family of functions
\[
r_Q(\Lambda,c_\lambda) = r_{Q|\bar{Q}}(c_\lambda)^{-1}r_{Q|\bar{Q}}(c_{\lambda+\Lambda/2}) 
\]
for $Q\in\P(M)$ and $\Lambda\in i\a^*_M$. The limit
\[
r^G_M(c_\lambda) = \lim_{\Lambda\to0} \sum_{Q\in \P(M)} r_Q(\Lambda,c_\lambda) \theta_Q(\Lambda)^{-1}
\]
is defined as a meromorphic function of $\lambda$. The global unnormalized weighted character, on the other hand,
\[
J_M(\dot\pi_\lambda,\dot{f}^1_z)=\tr(\J_M(\dot\pi_\lambda,P)\I_P(\dot\pi_\lambda,\dot{f}^1_z))
\]
is to be expressed in terms of the local normalized weighted characters
\[
J_L(\pi^L_\lambda,f)=\tr(\M_L(\pi^L_\lambda,{P_L})\I_{P_L}(\pi^L_\lambda,f)), \quad L\in \L(M), P_L\in\P(L).
\]
Since $\dot\pi$ is unramified outside of $V$, $\J_Q(\Lambda,\dot\pi_\lambda,P)$ is a scalar multiple of $\M_Q(\Lambda,\pi,P)$, namely
\[
\J_Q(\Lambda,\dot\pi_\lambda,P) = r_Q(\Lambda,c_\lambda,P)\mu_Q(\Lambda,c_\lambda,P)\M_Q(\Lambda,\pi_\lambda,P)
\]
according to \cite[p.207]{STF1}. We can then define the unramified character
\be
\label{rcb}
r^L_M(c_\lambda,b) = r^L_M(c_\lambda)b^V_M(c_\lambda), \qquad c\in C^V_\disc(M,\zeta).
\ee
It follows from \cite[Lemma 3.2]{STF1} and the absolute convergence of $b^V_M(c_\lambda)$ for Re$(s)$ large enough that for $c\in C_\text{disc}^V(M,\zeta)$, the function $r^L_M(c_\lambda,b)$ is an analytic function of $\lambda\in i\a^*_{M,Z}$ and
\[
\int_{i\a^*_{M,Z}/i\a^*_{G,Z}}r^L_M(c_\lambda,b)(1+||\lambda||)^{-N}d\lambda
\]
converges for $N$ large enough.

Let $\tilde\Pi_\disc(M,V,\zeta)$ be the preimage of $\Pi_\disc(M,V,\zeta)$ in $\Pi_\text{unit}(M(F_V),\zeta_V)$, and let $\Pi^G_\disc(M,V,\zeta)$ be the set of $i\a^*_{G,Z}$-orbits in $\tilde\Pi_\disc(M,V,\zeta)$.  There is a free action 
\[
\rho\to \rho_\lambda,\qquad \lambda\in i\a^*_{M,Z}/i\a^*_{G,Z}
\]
on $\Pi^G_\disc(M,V,\zeta)$ whose orbits can be identified with $\Pi_\disc(M,V,\zeta)$. Any element $\rho\in \Pi^G_\disc(M,V,\zeta)$ is an irreducible representation of $M(F_V)\cap G(F_V)^Z$, from which one can form the parabolically induced representation $\rho^G$ of $G(F_V)^Z$. Recall that the induction is defined by the relation 
\[
f_G(\rho^G) = f_M(\rho),\qquad f\in \C^\circ(G(F_V),\zeta_V),
\]
 with the adjoint restriction map $\pi\mapsto \pi_M$ from $\F(G(F_V),\zeta_V)$ to $\F(M(F_V),\zeta_V)$, such that
\be
\label{rhoG}
\sum_{\pi\in\Pi(G(F_V),\zeta_V)}c_M(\pi_M)d_G(\pi) = \sum_{\rho\in\Pi(M(F_V),\zeta_V)}c_M(\rho)d_G(\rho^G)
\ee
for any linear function $c_M$ on $\F(M(F_V),\zeta_V)$ and $d_G$ on $\F(G(F_V),\zeta_V)$. We then define $\Pi(G,V,\zeta)$ to be the union over $M\in \L$ and $\rho\in \Pi_\disc^G(M,V,\zeta)$ of irreducible constituents of $\rho^G$. It has a Borel measure $d\pi$ given by
\[
\int_{\Pi(G,V,\zeta)}h(\pi)d\pi = \sum_{M\in\L}|W_0^M|^{-1}|W_0^G|^{-1}\sum_{\rho\in\Pi_\disc(M,V,\zeta)}\int_{i\a^*_{M,Z}/i\a^*_{G,Z}}h(\rho^G_\lambda)d\lambda
\]
for any $h\in C_c(\Pi(G,V,\zeta))$.  We then define for any $\pi\in\Pi(G(F_V)^Z,\zeta_V)$, the spectral coefficient
\be
\label{arspi}
a^G_{r,s}(\pi) =  \sum_{L\in\L}|W^L_0||W^G_0|^{-1}\sum_{c\in C^V_\disc(M,\zeta)}a^M_\disc(\pi_M\times c)r^G_M(c,b),
\ee
where $\pi_M\times c$ is a finite sum of representations $\dot\pi$ in $\Pi_\text{unit}(M(\A),\zeta)$, and $a^M_\disc(\pi_M\times c)$ is the sum of corresponding values $a^G_\text{disc}(\dot\pi)$. It follows from the definitions that $a^G_{r,s}(\pi)$ is supported on the subset $\Pi(G,V,\zeta)$ of $\Pi(G(F_V)^Z,\zeta_V)$.

\begin{thm}
\label{Ispec}
Let $f\in \C^\circ(G,V,\zeta)$. Then the invariant linear form $I^r_s(f)$ has the spectral expansion
\be
\label{Irss}
I^r_s(f) = \sum_{M\in\L}|W^M_0||W^G_0|^{-1}\int_{\Pi(M,V,\zeta)}a^M_{r,s}(\pi)I_M(\pi,f)d\pi,
\ee
with the integrals converging absolutely.
\end{thm}

\begin{proof}

As in \eqref{fzeta2}, we would like a parallel expansion for the linear form
\[
J^r_s(f) = J^\zeta(\dot f^1_{r,s}) = \int_{Z(F)\bs Z(\A)^1} J(\dot f^1_{r,s,z}) \zeta(z)dz
\]
where $\dot f^1_{r,s}$ is any function in in $\C^\circ(G(\A_F)^1)$ whose projection onto $\C^\circ(G(\A_F)^Z,\zeta)$ equals $\dot f^r_s = f \times b^V$. It then follows from Proposition \ref{eis} that $J^r_s(f)$ has an expansion
\[
\int_{Z(F)\bs Z(\A)^1}\sum_{M\in\L}|W^M_0||W^G_0|^{-1}\sum_{\dot{\pi}\in\Pi_\disc(M)}\int_{i\a^*_{M,Z}/ i\a^*_{G,Z}} a^M_\disc(\dot{\pi}_\lambda)J_M(\dot{\pi}_\lambda,\dot{f}^1_{r,s,z})\zeta(z)d\lambda \ dz.
\]
Now the outer integral annihilates the contribution of $\dot\pi$ in the complement of $\Pi_\disc(M,\zeta)$ in $\Pi_\disc(M)$, so the expression simplifies to
\[
J^r_s(f) =\sum_{M\in\L}|W^M_0||W^G_0|^{-1}\sum_{\dot{\pi}\in\Pi_\disc(M,\zeta)}\int_{i\a^*_{M,Z}/ i\a^*_{G,Z}} a^M_\disc(\dot{\pi}_\lambda)J_M(\dot{\pi}_\lambda,\dot{f}_{r,s})d\lambda.
\]

We have to express $J_M(\dot{\pi}_\lambda,\dot{f})$ in terms of local normalized weighted characters.  Applying the splitting formula
\[
\J_M(\dot\pi_\lambda,P) = \sum_{L\in\L(M)}r^L_M(c_\lambda)\M_L(\pi^L_\lambda,P)
\]
in \cite[p.208]{STF1}, we have for the choice of test function function $\dot{f}^r_s$ the relation
\[
J_M(\dot\pi_\lambda,\dot f^r_s) = \sum_{L\in\L(M)}r^L_M(c_\lambda,b) J_M(\pi^L_\lambda,f),
\]
which vanishes unless $\dot\pi$ is unramified outside of $V$. Writing $\dot\pi = \pi\times c$, we can replace the sum over $\dot\pi$ with a sum over $\pi$ in $\Pi_\text{disc}(M,V,\zeta)$ and $c\in C^V_\text{disc}(M,\zeta)$. From the definition of the spectral coefficients, we write
\[
a^L_{r,s}(\pi^L_\lambda) = \sum_{c \in \C^V_\text{disc}(M,\zeta)}a^M_\text{disc}(\pi_\lambda\times c_\lambda)r^L_M(c_\lambda,b )
\]
for any $\lambda\in i\a^*_{M,Z}$ in general position, ignoring sets of measure zero. We can therefore rewrite our expansion as
\[
\sum_{L\in\L}|W^L_0||W^G_0|^{-1}\sum_{M\in\L^L}|W^M_0||W^L_0|^{-1}\sum_{\pi\in\Pi_\text{disc}(M,V,\zeta)}\int_{i\a^*_{M,Z}/i\a^*_{G,Z}} a^L_{r,s}(\pi^L_\lambda) J_L(\pi^L_\lambda,f)d\lambda.
\]
The coefficient $a^L(\pi^L_\lambda)$ and the integral 
\[
\int_{i\a^*_{L,Z}/i\a^*_{G,Z}}  J_L(\pi^L_{\lambda+\Lambda},f)d\Lambda
\]
depend only on the image of $\lambda$ in $i\a^*_{M,Z}/i\a^*_{L,Z}$, hence on the restriction of $\pi^L_\lambda$ to $L^V_Z$. Writing $\pi$ for this restriction, which runs over $\Pi(L,V,\zeta)$, we arrive at 
\[
J(\dot f^r_s)= \sum_{L\in\L}|W^L_0||W^G_0|^{-1}\int_{\Pi(L,V,\zeta)}a^L(\pi)J_L(\pi,f)d\pi .
\]

Again, to convert the expansion for $J(f^r_s)$ into an expansion of $I(f^r_s)$,  we assume inductively that the expansion \eqref{Irss} holds if $G$ is replaced by any proper Levi $L\in \L^0$. Then using the expansion we have just obtained for $J^r_s(f)$, we see that $I^r_s(f)$ is equal to
\[
 \sum_{M\in\L}|W^M_0||W^G_0|^{-1}\int_{\Pi(M,V,\zeta)}a^M_{r,s}(\pi)\left(J_M(\pi,f) - \sum_{L\in \L(M)}\hat{I}^L_M(\pi,\phi_L(f))\right)d\pi.
\]
By the definition of $I_M(\pi,f)$ then this is equal to 
\[
\sum_{M\in\L}|W^M_0||W^G_0|^{-1}\int_{\Pi(M,V,\zeta)}a^M_{r,s}(\pi)I_M(\pi,f)d\pi
\]
as required.
\end{proof}

Putting Theorems \ref{Igeom} and \ref{Ispec} together, we have an invariant trace formula that is valid for $\dot f^r_s = f \times b$. 

\bibliography{BESL2}
\bibliographystyle{alpha}
\end{document}